\documentclass{amsart}
\usepackage{amsmath,amssymb}
\usepackage[margin=1in]{geometry}
\usepackage{graphicx}
\usepackage{tikz}
\usepackage{stanli}
\usetikzlibrary{decorations.pathreplacing}
\usepackage{url}
\usepackage{mathtools}
\usepackage[shortlabels]{enumitem}
\usepackage{algorithm,algpseudocode}
\usepackage{MnSymbol}
\usepackage[differential]{boldtensors}

\usepackage[utf8]{inputenc}
\usepackage[english]{babel}
\usepackage{tabularx}
\usepackage{ulem}
\usepackage{subcaption}

\makeatletter
\newcommand{\multiline}[1]{%
  \begin{tabularx}{\dimexpr\linewidth-\ALG@thistlm}[t]{@{}X@{}}
    #1
  \end{tabularx}
}
\makeatother

\usepackage{xcolor}

\DeclareFontFamily{U}{MnSymbolC}{}
\DeclareSymbolFont{MnSyC}{U}{MnSymbolC}{m}{n}
\DeclareFontShape{U}{MnSymbolC}{m}{n}{
    <-6>  MnSymbolC5
   <6-7>  MnSymbolC6
   <7-8>  MnSymbolC7
   <8-9>  MnSymbolC8
   <9-10> MnSymbolC9
  <10-12> MnSymbolC10
  <12->   MnSymbolC12}{}
\DeclareMathSymbol{\intprod}{\mathbin}{MnSyC}{'270}

\newtheorem{theorem}{Theorem}[section]

\newtheorem{proposition}{Proposition}[section]

\theoremstyle{definition}

\theoremstyle{definition}
\newtheorem{remark}[theorem]{Remark}

\DeclareMathOperator*{\argmin}{arg\,min}
\DeclareMathOperator*{\baroplus}{\bar{\oplus}}
\DeclareMathOperator*{\vect}{{\mathrm{vec}}}

\newcommand{\RN}[1]{
    \textup{\uppercase\expandafter{\romannumeral#1}}
}
\newcommand{\lr}[1]{\left(#1\right)}

\newcommand{\bb}[1]{\mathbf{#1}}

\newcommand{\IP}[2]{\left\langle #1,#2\right\rangle}

\newcommand{\nn}[1]{\left|#1\right|}

\newcommand{\Exterior}{\mathchoice{{\textstyle\bigwedge}}%
    {{\bigwedge}}%
    {{\textstyle\wedge}}%
    {{\scriptstyle\wedge}}}

\let\originalleft\left
\let\originalright\right
\renewcommand{\left}{\mathopen{}\mathclose\bgroup\originalleft}
\renewcommand{\right}{\aftergroup\egroup\originalright}

\title{Variationally Consistent Hamiltonian Model Reduction}
\author{Anthony Gruber, Irina Tezaur}

\begin{document}

\begin{abstract}
    Though ubiquitous as first-principles models for conservative phenomena, Hamiltonian systems present numerous challenges for model reduction even in relatively simple, linear cases.  Here, we present a method for the projection-based model reduction of canonical Hamiltonian systems that is variationally consistent for any choice of linear reduced basis: Hamiltonian models project to Hamiltonian models.  Applicable in both intrusive and nonintrusive settings, the proposed method is energy-conserving and symplectic, with error provably decomposable into a data projection term and a term measuring deviation from canonical form.  Examples from linear elasticity with realistic material parameters are used to demonstrate the advantages of a variationally consistent approach, highlighting the steady convergence exhibited by consistent models where previous methods reliant on inconsistent techniques or specially designed bases exhibit unacceptably large errors. 
    
    \vspace{0.5pc}

    
\end{abstract}

\maketitle

\section{Introduction}\label{sec:intro}
Hamiltonian systems represent a simple description of purely conservative dynamics, and are often viewed as idealized limits of realistic physical systems, e.g., the Euler equations as the zero-viscosity limit of Navier-Stokes' equation.  Despite their relative simplicity, the subclass of canonical Hamiltonian systems still exhibits many challenges for model reduction. This is especially the case for  projection-based reduced order models (ROMs), where the governing equations are projected onto a data-driven basis for full-order model (FOM) solutions.  It is straightforward to show that a na\"{i}ve basis projection will not preserve Hamiltonian structure in the ROM (c.f. \cite{peng2016symplectic}), meaning that the reduced ordinary differential equation (ODE) system, which updates the solution coefficients, will not be energy-conserving.  This means that approximations built from a na\"ive projection-based ROM will not preserve critical invariants of the FOM, leading to artificial instabilities and other unphysical behavior in time.  

To describe this more concretely, consider a canonical Hamiltonian system in the state $~x=\begin{pmatrix}~q&~p\end{pmatrix}^\intercal\in\mathbb{R}^N$, where $N = 2M\in\mathbb{R}$, comprised of position and momentum variables,
\begin{equation}\label{eq:fom}
    \dot{~x} = ~J\nabla H(~x), \quad ~x_0 = ~x(0),
\end{equation} 
where $~J = -~J^\intercal\in\mathbb{R}^{N\times N}$ is the full-rank and skew-symmetric symplectic matrix $~J = [~0\,\,~I; -~I\,\,~0]$ and $H:\mathbb{R}^n\to \mathbb{R}$ is a Hamiltonian functional \cite{Marsden_BOOK_1998}.  
 Denote the flow map of \eqref{eq:fom} by $~\varphi(t,~x_0) = ~\varphi_t(~x_0)$, so that $~\varphi_t(~x_0)$ represents the position along the integral curve of \eqref{eq:fom} beginning at $~x_0$ after time $t$.  The standard intrusive Galerkin ROM based on Proper Orthogonal Decomposition (POD) \cite{Sirovich:1987, Holmes:1996} utilizes a matrix of snapshots $~X = \{~\varphi\lr{t_i, ~x_0}\}_{i=0}^K \in \mathbb{R}^{N\times K}$, collected offline, to construct the POD basis matrix $~U\in\mathbb{R}^{N\times n}$, $n=2m$, through the truncated Singular Value Decomposition (SVD) $~X \approx ~U~\Sigma~V^\intercal$.  If $\tilde{~x}=~U\hat{~x}$ is an approximate state given by Galerkin projection onto the span of $~U$, it follows that $~U^\intercal\dot{~\varphi}_t\lr{~U\hat{~x}_0} = \dot{\hat{~x}}$ is the projected vector field and the differential equation
\begin{equation}\label{eq:galROM}
    \dot{\hat{~x}} = ~U^\intercal~J\nabla H\lr{~U\hat{~x}}, \quad \hat{~x}_0 = ~U^\intercal~x_0,
\end{equation}
defines the natural Galerkin ROM corresponding to \eqref{eq:fom}.  While this projection accomplishes the primary goal of restricting the dynamics to a low-dimensional subspace of the FOM data, it is easy to check that the Hamiltonian structure guaranteeing Lyapunov stability of the system is not preserved through this procedure, since $\lr{~U^\intercal~J}^\intercal = -~J~U \neq -~U^\intercal~J$.  This has obvious negative consequences: besides the loss of long-time stability, there is no guarantee of energy conservation or symplectic structure anymore. To see this, let $\hat{H}=H\circ\tilde{~x}$ be the reduced Hamiltonian. 
 Then, $\nabla \hat{H} = ~U^\intercal\nabla H$, and so the instantaneous energy satisfies 
\[\dot{\hat{H}} = \dot{\hat{~x}}\cdot\nabla \hat{H} = ~U^\intercal~J\nabla H \cdot ~U^\intercal\nabla H = ~P_U\nabla H\cdot~J\nabla H,\]
which likely does not vanish in view of the projection $~P_U:=~U~U^\intercal$ onto the column span of $~U$.  Additionally, the reduced dynamics are no longer symplectic: there is no closed and nondegenerate differential two-form preserved along the flow, meaning that volumes in phase space are unnaturally distorted by this evolution.  It has been shown that this lack of Hamiltonian structure in the straightforward Galerkin ROM leads to low-order systems which are unacceptably sensitive and cannot forecast beyond their training range, even in the case of periodic dynamics \cite{peng2016symplectic,gong2017structure,gruber2023canonical}. 

This unfortunate deficiency of projection-based ROMs in the Hamiltonian case has been well studied, and several solutions are available at present, each with their own benefits and drawbacks.  One popular approach involves utilizing a symplecticity-preserving approach such as the Proper Symplectic Decomposition (PSD) to compute the reduced basis $~U$, which includes basis algorithms such as the cotangent lift and complex SVD \cite{peng2016symplectic, Hesthaven:2021}. In each case, the basis $~U\in\mathbb{R}^{N\times n}$ is constructed to satisfy the equivariance condition $~U^\intercal~J = ~J_n~U^\intercal$, where $~J_n\in\mathbb{R}^{n\times n}$ is the canonical symplectic matrix in $n$ dimensions \cite{Marsden_BOOK_1998}.  
This admits the pleasing interpretation of projecting the canonical Hamiltonian FOM \eqref{eq:fom} directly to a canonical Hamiltonian ROM, which is obviously property preserving.  It also enables the direct construction of this Hamiltonian ROM through straightforward Galerkin projection, since the evolution of the reduced state becomes
\begin{align}\label{eq:cotliftROM}
    \dot{\hat{~x}} = ~U^\intercal\dot{\tilde{~x}} = ~U^\intercal~J\nabla H\lr{\tilde{~x}} = ~J_n~U^\intercal\nabla H\lr{~U\hat{~x}} = ~J_n\nabla\hat{H}\lr{\hat{~x}},
\end{align}
which is a low-order Hamiltonian system in canonical form.  
Sometimes called the symplectic lift \cite{peng2016symplectic}, this technique effectively converts the problem of preserving Hamiltonian structure in projection-based ROM with constructing an informative $~J$-equivariant reduced basis $~U\in\mathbb{R}^{N\times n}$ for the dynamics.  
However, as useful as the symplectic lift is for constructing Hamiltonian ROMs, it breaks down when the bases used are not informative enough to accurately represent system dynamics.  This is particularly noticeable in realistic mechanics simulations: when the values $~q,~p$ of position and momentum are separated by orders of magnitude, the basis $~U$ resulting from the above procedures is often not informative enough to produce a useful Hamiltonian ROM of the form \eqref{eq:cotliftROM}, and this effect persists even when these scales are equalized in the training data (c.f. Section~\ref{sbsec:linelastic}).
Moreover, the symplectic lift is only feasible for canonical Hamiltonian systems where the symplectic matrix $~J \neq ~J(~x)$ is a known constant, which limits its applicability to many systems of practical interest \cite{gruber2023canonical}.

The issues with constructing the equivariant basis necessary for symplectic lift described above led to another line of work involving ROMs which preserve the Hamiltonian of structure of \eqref{eq:fom} \cite{gong2017structure}.  In this case, the goal was to allow the use of an arbitrary reduced basis $~U$, so that the potentially high approximation error of the symplectic lift could be sidestepped and Hamiltonian ROMs could be extended to the noncanonical case.  This led to a nice strategy in \cite{gong2017structure} based on the observation that the equivariance condition $~U^\intercal~J=~J_n~U^\intercal$ is sufficient but not necessary for a  Hamiltonian ROM.  Treating this condition instead as an overdetermined system in the reduced symplectic matrix $\hat{~J}$, the skew-symmetric least-squares solution $\hat{~J}=~U^\intercal~J~U$ leads directly to the ROM
\begin{equation}\label{eq:symROM}
    \dot{\hat{~x}} = ~U^\intercal~J~U~U^\intercal\nabla H\lr{~U\hat{~x}} = \hat{~J}\nabla\hat{H}\lr{\hat{~x}}, \quad \hat{~x}_0 = ~U^\intercal~x_0.
\end{equation}
Since it is both energy-conserving and symplectic, this has the advantage of being a valid, fully determined Hamiltonian ROM for almost all bases $~U$ (c.f. Proposition~\ref{prop:nondegen}).  These properties follow since the instantaneous energy satisfies
\begin{align*}
    \dot{\hat{H}}=\dot{\hat{~x}}\cdot\nabla\hat{H} = \hat{~J}\nabla\hat{H}\cdot\nabla\hat{H}=-\nabla\hat{H}\cdot\hat{~J}\nabla\hat{H}=0,
\end{align*}
and the skew-symmetric matrix $\hat{~J}$ is constant in $~x$, meaning it satisfies the Jacobi identity equivalent to closure of the governing symplectic two-form.  The latter condition ensures that volume densities are instantaneously conserved in phase space and no artificial dissipation can be introduced into the scheme (c.f. \cite{caligan2016conservative}).  While the ROM \eqref{eq:symROM} is not canonical unless $~U$ is $~J$-equivariant, it is in some sense the reduced-order Hamiltonian system nearest to $\eqref{eq:galROM}$ for a given $~U$, making it advantageous as a structure-preserving alternative.  Moreover, the least-squares technique generating \eqref{eq:symROM} extends directly to noncanoncial Hamiltonian systems, since it does not require the construction of a special reduced basis.  On the other hand, there is clearly an additional column space projection $~P_U=~U~U^\intercal$ in \eqref{eq:symROM} which arises ``extraneously'' via the least-squares projection of $~J$ on $~U$.  It turns out that this additional projection in the reduced vector field $\dot{\hat{~x}}$ is enough to significantly harm ROM accuracy in some cases of interest, which will be seen in Section~\ref{sbsec:linelastic}.  Fundamentally, this occurs because the ROM \eqref{eq:symROM} fails to be variationally consistent with the original Hamiltonian FOM: the Hamiltonian variational principle obeyed by the reduced state $\hat{~x}$ is not the Galerkin projection of the Hamiltonian variational principle obeyed by the original state $~x$ (see Figure \ref{fig:no-commute-cartoon}), and therefore the Galerkin projection of solutions to \eqref{eq:fom} fails to satisfy the reduced equations \eqref{eq:symROM}. This necessitates a different model, developed in Sections~\ref{sec:method} and \ref{sec:opinf}, which does not require inconsistent modifications to produce a reduced-order Hamiltonian vector field.


\subsection{Related work}
Before discussing the contributions of this work in detail, it is prudent to mention several other approaches recently proposed for maintaining Hamiltonian structure in intrusive as well as nonintrusive projection-based ROMs.  Beginning with the intrusive case, Sockwell develops a Hamiltonian structure-preserving approach in \cite{Sockwell_THESIS_2019} based on the work of \cite{gong2017structure}, which targets noncanonical systems and includes a useful centering procedure for preserving linear Casimir invariants such as mass and momentum.  Buchfink \textit{et al.} derive a globally optimal symplectic reduced basis for ``canonizable'' Hamiltonian systems with a periodic solution in the sense of PSD \cite{Buchfink:2022}, and further extensions of PSD are offered by Hesthaven \textit{et al.} in \cite{Afkham:2017, Hesthaven:2021}.  More specifically, the authors of \cite{Afkham:2017} propose a greedy extension of symplectic basis generation, which promises to be more cost-effective than traditional POD and exhibits an exponential convergence rate.  The related work \cite{Hesthaven:2021} proposes a dynamical symplectic reduced basis method for low-rank Hamiltonian systems based on evolving the reduced space along a trajectory locally constrained to lie on the tangent space of the high-dimensional dynamics. Beyond purely Hamiltonian dynamics, extensions of structure-preserving model reduction methods to problems with dissipation have been the subject of several recent works including \cite{Afkahm:2019} for the case of Hamiltonian systems with dissipation and \cite{Gruber:2023} for the case of metriplectic systems which combine energy conservation with entropy generation.


In addition to intrusive model reduction approaches, recent years have seen the emergence of data-driven structure-preserving operator inference (OpInf) ROMs.  The premise of data-driven OpInf, first introduced by Peherstorfer and Willcox in 
\cite{peherstorfer2016data}, is that reduced operators in a projection-based ROM can be learned nonintrusively
(i.e., without access to the FOM operators or code) given a set of FOM snapshots by solving an optimization problem.  In this direction, several recent papers have devised OpInf algorithms which preserve Hamiltonian structure.  The first of these works \cite{sharma2022hamiltonian}, termed H-OpInf, preserves canonical Hamiltonian structure in OpInf ROMs built using a ``cotangent lift'' (i.e., PSD) reduced basis. This approach is extended to quadratic reduced bases in \cite{Sharma:2023}, towards improving accuracy and efficiency for problems exhibiting a slowly decaying Kolmogorov $n$-width \cite{Pinkus:1985}.    
Again focusing on canonical Hamiltonian systems, this second approach features a quadratic manifold cotangent lift algorithm for reduced basis construction as well as a blockwise-quadratic approximation of each of the states, $~q$ and $~p$.  This line of work is extended in \cite{gruber2023canonical} to accommodate both canonical and the more general case of noncanonical Hamiltonian systems, enabling linear Casimir-preserving Hamiltonian ROMs built with any choice of linear reduced basis.  Further extensions of Hamiltonian structure-preserving nonintrusive model reduction methods to nonlinear manifold ROMs using weakly symplectic autoencoders are developed in \cite{Buckfink:2021}, and an approach for applying OpInf to nonlinear Hamiltonian problems with polynomial non-linearities is developed in \cite{Benner:2020}. 
A more general approach for building structure-preserving ROMs for problems with symmetric negative-definite operators and linear or quadratic nonlinearities is presented in \cite{Sawant:2023}, where a Tikhonov regularization-inspired approach is introduced to improve the stability of the underlying optimization problem.  Further extensions to nonlinear Hamiltonian systems are explored in \cite{Yildiz:2024}, which hypothesizes that general Hamiltonian systems can be transformed to higher-dimensional nonlinear systems with cubic Hamiltonians through a lifting procedure \cite{Qian:2020}. Finally, a generalization of \cite{sharma2022hamiltonian} to systems that are either conservative or dissipative is proposed in \cite{Geng:2024}, which focuses on the preservation of gradient structure in OpInf ROMs.  Notably, the authors of this work also derive an \textit{a priori} error estimator for their OpInf ROMs which inspired Theorem~\ref{thm:newop} presented here.

Before concluding this literature overview, it is worth noting that (under mild regularity conditions) the canonical Hamiltonian formulation \eqref{eq:fom} has an analogous Lagrangian formulation, in which the first-order-in-time system \eqref{eq:fom} for the state $~x=\begin{pmatrix}~q&~p\end{pmatrix}^\intercal\in\mathbb{R}^{2M}$ consisting of positions and momenta is replaced with an equivalent second-order-in-time system of so-called Euler-Lagrange equations for the unknown positions $~q \in \mathbb{R}^M$.  
For linear systems with no damping, the Euler-Lagrange equations take the form 
\begin{equation} \label{eq:Lagrangian}
~M \ddot{~q} + ~K {~q} = ~f,
\end{equation}
where $~M, ~K \in \mathbb{R}^{M\times M}$ are the symmetric positive definite mass and stiffness matrices respectively, and $~f \in \mathbb{R}^M$ is an external force vector.  The Lagrangian formulation is preferred in some applications, e.g., solid dynamics, since it contains half as many degrees of freedom relative to the analogous Hamiltonian formulation; hence, a number of researchers have also developed model reduction methods that preserve Lagrangian structure.  Remarkably,   
Lall \textit{et al.} showed in \cite{Lall:2003} that performing a Galerkin projection of the second-order-in-time Euler-Lagrange equations preserves Lagrangian structure for linear and nonlinear problems, provided no hyper-reduction is employed, making structure preservation simpler in this case\footnote{For maintaining Lagrangian and Hamiltonian structure in ROMs for nonlinear Lagrangian systems, several structure-preserving hyper-reduction approaches have been developed, e.g., reduced basis sparsification (RBS) and matrix gappy POD \cite{Carlberg:2015}, Energy-Conserving Sampling and weighting \cite{Farhat:2015} and the Symplectic Discrete Empirical Interpolation Method (SDEIM) \cite{peng2016symplectic}; as this work focuses on linear problems, these methods are not discussed at length here.}.  Approaches for preserving Lagrangian structure in OpInf models, termed L-OpInf ROMs, have been considered by Sharma \textit{et al.} in the recent works \cite{Sharma:2024} and \cite{Sharma:2024_2}. 
While the basic premise of \cite{Sharma:2024} is to formulate and solve a constrained optimization problem for the reduced operators $\hat{~M}$ and $\hat{~K}$, \cite{Sharma:2024_2} proposes a two-step approach that first learns the best-fit linear Lagrangian ROM via Lagrangian operator inference and subsequently applies a structure-preserving machine learning method to learn nonlinearities in the reduced space.

\begin{remark}\label{rem:sympint}
    While continuous models predicated on Hamiltonian/Lagrangian variational principles are automatically conservative, practical FOMs/ROMs need to be advanced with a symplectic time-integrator in order to be energy-conserving along their trajectories.  This work employs the Newmark-Beta integrator for Lagrangian systems \eqref{eq:Lagrangian} and the implicit midpoint method (trapezoid rule) for Hamiltonian systems \eqref{eq:fom}, both of which ensure that energy is conserved at the fully discrete level.
\end{remark}

\subsection{Contributions}


The present work aims to combine the variational consistency of the $~J$-equivariant PSD approach \eqref{eq:cotliftROM} with the flexibility of the least-squares approach \eqref{eq:symROM} to yield a Hamiltonian ROM which outperforms previous methods on moderate-to-large scale Hamiltonian problems of interest.  In particular, the goal is to develop a Hamiltonian ROM which is variationally consistent with the governing Hamiltonian FOM \eqref{eq:fom}, allows for the use of any desired trial space basis $~U$, and which can be deployed in both intrusive and nonintrusive settings.  This necessitates a new technique, presented here and applicable to ROMs constructed using both intrusive and nonintrusive (OpInf) methods.  Its core idea is simple: while the trial space basis should be treated as fixed for optimal dimension reduction, there is flexibility in the space of test functions used to generate the ROM.  It is shown that an inspired choice of test space yields a Hamiltonian ROM which is valid for any canonical Hamiltonian FOM and requires nothing more than Petrov-Galerkin orthogonality.  
Specifically, this work contributes the following:

\begin{itemize}
    \item A procedure to generate variationally consistent Hamiltonian ROMs from canonical Hamiltonian FOMs with minimal assumptions on the trial space basis, applicable in both intrusive and nonintrusive settings.
    \item Interpretable error estimates showing the decomposition of ROM error into two distinct parts: basis projection error and deviation-from-canonicity error.
    \item Seamless incorporation of existing ROM technology, such as a centering procedure which preserves linear invariant quantities, as well as re-projection which ensures that nonintrusive OpInf ROMs pre-asymptotically recover the performance of their intrusive counterparts.
    \item 3D solid mechanical examples from linear elasticity showing improvements over existing state of the art methods for both intrusive and nonintrusive OpInf ROMs.
\end{itemize}

The remainder of this paper is organized as follows. Section \ref{sec:method} describes the core framework, leading to a novel variationally consistent intrusive Hamiltonian ROM which accommodates any desired linear reduced basis.  Section \ref{sec:opinf} extends this approach to nonintrusive OpInf ROMs, culminating in a straightforward algorithm for inferring the relevant operators from snapshot data.  An error analysis of the proposed intrusive and nonintrusive ROMs is performed in Section \ref{sec:analysis}, where interpretable error estimates for these approaches are derived and discussed.  The numerical performance of the proposed ROMs is further examined in Section \ref{sec:numerics}, focused on realistic examples from linear elasticity.  Finally, a concluding summary and a discussion of future work is presented in Section \ref{sec:conc}.

\section{Variationally Consistent Intrusive Hamiltonian ROM}\label{sec:method}



The present approach is motivated by the standard coordinate-free formulation of Hamiltonian systems, of which more details can be found in, e.g., \cite{Marsden_BOOK_1998}.  Recall that a canonical Hamiltonian system can be alternatively described as a symplectic gradient flow, defined in terms of the symplectic structure on the even-dimensional $\mathbb{R}^N$ ($N=2M$) 
induced by a canonical closed and nondegenerate two-form $\omega\in\Exterior^2\mathbb{R}^N$.  Particularly, \eqref{eq:fom} is equivalent to $\dot{~x}= ~X_H$, where $~X_H$ is the symplectic gradient of $H$ defined through the differential of the Hamiltonian function, $dH = ~X_H\intprod\omega \coloneqq \omega\lr{~X_H, -}$, which is a mapping from tangent vector fields on $\mathbb{R}^N$ to tangent vector fields on $\mathbb{R}$.  To see why this is an equivalence, notice that the negative of the canonical K\"{a}hler structure on $\mathbb{R}^N$ induces a smooth matrix field $~J\in\mathbb{R}^{N\times N}$ satisfying $\omega\lr{~v,~w} = ~v\cdot~J~w$ for any $~v,~w\in T_{~x}\mathbb{R}^N\simeq \mathbb{R}^N$ at any point $~x\in\mathbb{R}^N$, which coincides with the constant symplectic matrix $~J$ in \eqref{eq:fom}.  This implies that for any $~v\in\mathbb{R}^N$ the differential of the Hamiltonian satisfies
\[dH(~v) = ~v\cdot\nabla H = ~X_H\cdot~J~v = ~v\cdot~J^\intercal~X_H,\]
and therefore the symplectic gradient $~X_H$ is given by 
\[~X_H = ~J^{-\intercal}\nabla H = ~J\nabla H,\]
confirming that the evolution of the state is given by $\dot{~x}=~X_H=~J\nabla H$ in agreement with \eqref{eq:fom}.
\begin{remark}
    Note that, here, the implied definiteness condition $\omega\lr{~v,~J~v} = -\nn{~v}^2 <0$ for vectors $~v\neq~0$ is opposite to what is required for the canonical K\"{a}hler form on $\mathbb{R}^{N}$.  Since either choice leads to the same mathematical structure, this choice is made so that signs are consistent with the FOM $\eqref{eq:fom}$.
\end{remark}

While the expression $\dot{~x}=~X_H$ in terms of the symplectic gradient $~X_H$ may look like a straightforward rewrite of the FOM \eqref{eq:fom}, the advantage of this expression lies in its ability to convey a simple requirement for any Hamiltonian ROM derived from a canonical Hamiltonian FOM.  Particularly, it follows that a reduced Hamiltonian ROM in terms of the approximate state $\tilde{~x}=~U\hat{~x}$ must obey the low-order evolution equation
\begin{align*}
    \dot{\hat{~x}} = ~X_{\hat{H}} = \hat{~J}^{-\intercal}\nabla\hat{H},
\end{align*}
involving the reduced symplectic gradient $~X_{\hat{H}}\in\mathbb{R}^n$, $n=2m\ll N$, 
corresponding to some reduced Hamiltonian function $\hat{H}:\mathbb{R}^n\to\mathbb{R}$, and some closed and nondegenerate symplectic two-form  $\hat{\omega}\in\Exterior^2\mathbb{R}^n$ defining the reduced symplectic matrix $\hat{~J}=~U^\intercal~J~U\in\mathbb{R}^{n\times n}$.  Importantly, notice that this reduced symplectic form $\hat{\omega}$ need not coincide with the canonical symplectic form on $\mathbb{R}^{n}$, so that $\hat{~J}=~B^\intercal~J_n~B$ can be conjugate to the canonical symplectic matrix $~J_n\in\mathbb{R}^{n\times n}$ by an \textit{unknown} invertible matrix $~B\in\mathbb{R}^{n\times n}$.  In particular, this means that the transpose of the reduced symplectic matrix $\hat{~J}^\intercal=-\hat{~J}$ may not be the inverse of $\hat{~J}$ anymore, i.e. $\hat{~J}^\intercal\neq\hat{~J}^{-1}$. This represents the fact that the reduced symplectic form $\hat{\omega}$, and hence the reduced symplectic matrix $\hat{~J}$, are allowed to be noncanonical on the reduced space $\mathbb{R}^n$.  This constrasts significantly with the case of $~J$-equivariant reduced bases $~U$, where $\hat{~J}=~J_n$ and so the canonicity condition $\hat{~J}^{-1}=\hat{~J}^\intercal$ is satisfied by construction.

In view of this discussion, there are two main ingredients to the variationally consistent Hamiltonian ROM presented here.  The first, and most essential, is a way to manufacture a reduced symplectic gradient $~X_{\hat{H}}\in\mathbb{R}^n$ from the full order symplectic gradient $~X_H\in\mathbb{R}^N$ using only projection onto orthonormal reduced bases.  This will be accomplished via an appropriately chosen Petrov-Galerkin projection, which effectively replaces the symplectic matrix $~J$ on the right-hand side of \eqref{eq:fom} with a reduced inverse-transpose $\hat{~J}^{-\intercal}$, making the resulting ROM variationally consistent with the FOM by construction.  This is remarkably different from the case of standard Galerkin projection onto the trial space spanned by the columns of $~U$, where the resulting ROM will not be Hamiltonian unless the reduced basis $~U$ is $~J$-equivariant (c.f. Figure~\ref{fig:no-commute-cartoon}).  The second ingredient is a useful but nonessential centering procedure in the trial space basis $~U$, which is commonly used in POD ROMs and ensures that the value of the Hamiltonian function $H$ is correct at the beginning of the ROM trajectory, and hence will remain correct throughout its evolution.  While this particular centering has previously been used only for intrusive ROMs, its applicability is extended to the nonintrusive OpInf case here (see Section \ref{sbsec:nonint-centering}), which may be of independent interest.  
Combined with the idea of re-projection \cite{peherstorfer2020sampling}, this enables a nonintrusive Hamiltonian OpInf ROM that pre-asymptotically recovers the intrusive ROM, enabling its use in the case that only snapshot data are available.

\begin{figure}[htb]
    \centering
    \includegraphics[width=0.7\textwidth]{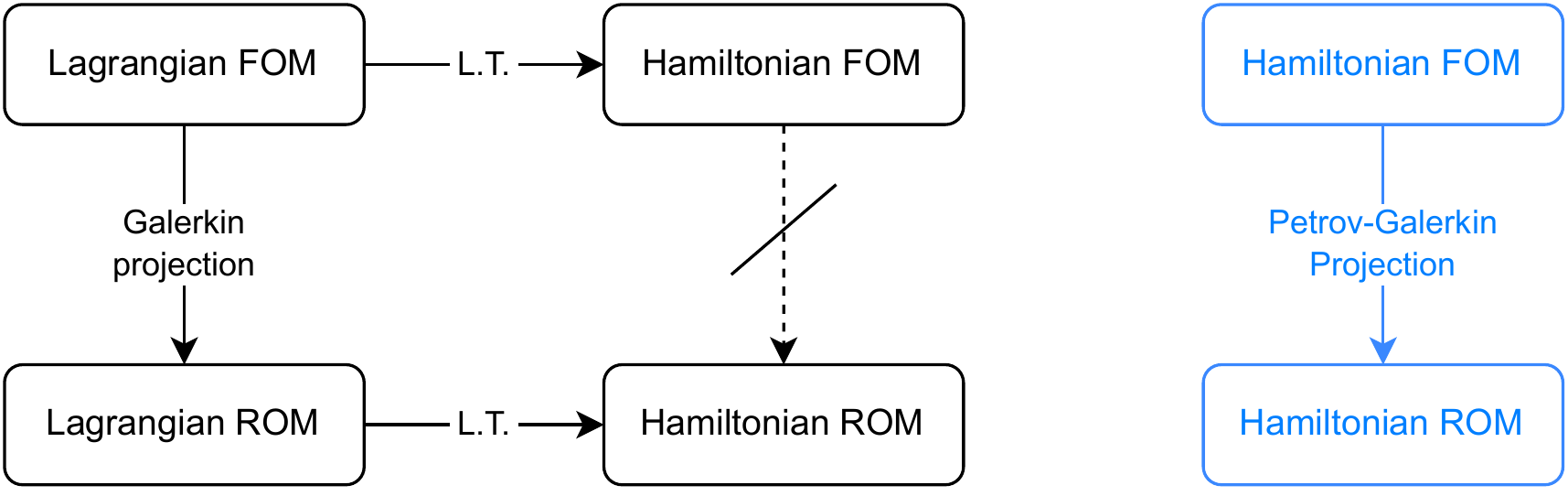}
    \caption{A diagram illustrating the behavior of Hamiltonian/Lagrangian FOMs under Galerkin projection.  While (canonical) Hamiltonian and Lagrangian systems which are ``regular'' are equivalent under the Legendre transformation (L.T.) at both the FOM and ROM levels, this equivalence does not commute with Galerkin projection onto a general reduced basis $~U\in\mathbb{R}^{N\times n}$.  Conversely, the variationally consistent ROM presented here projects Hamiltonian FOMs to Hamiltonian ROMs in every case.}
    \label{fig:no-commute-cartoon}
\end{figure}





\subsection{Variationally consistent Petrov-Galerkin projection}
The particular basis projection employed here exploits the fact that the symplectic structure of canonical Hamiltonian systems is nondegenerate everywhere, meaning that the symplectic matrix $~J$ has full matrix rank.  Therefore, given any trial space basis $~U\in\mathbb{R}^{N\times n}$, the columns of the ``rotated'' matrix $~J~U$ form a valid test space basis of rank $n$, meaning the Petrov-Galerkin projection $\lr{~J~U}^\intercal\dot{~\varphi}_t\lr{~U\hat{~x}} = ~U^\intercal~J^\intercal~U\dot{\hat{~x}}=\hat{~J}^\intercal\dot{\hat{~x}}$ leads to a valid Hamiltonian ROM provided that the reduced symplectic matrix $\hat{~J} = ~U^\intercal~J~U$ remains full rank.
Noting that $~J^\intercal~J = -~J^2 = ~I$, this ROM is expressible as
\begin{equation}\label{eq:newROM}
   \hat{~J}^\intercal\dot{\hat{~x}} = ~U^\intercal\nabla H\lr{~U\hat{~x}} = \nabla\hat{H}(\hat{~x}), \quad \hat{~x}_0 = ~U^\intercal~x_0,
\end{equation}
which is equivalent, at least formally, to the expression $\dot{\hat{~x}}=\hat{~J}^{-\intercal}\nabla\hat{H}\lr{\hat{~x}}$ from before.  However, it remains to be sure that the differential equation \eqref{eq:newROM} uniquely specifies the evolution of the reduced order state $\hat{~x}$, since the full rank condition on $\hat{~J}$ is not obvious.  The following result shows that this condition is indeed satisfied for almost every orthonormal trial space basis $~U$ of even dimension that may be constructed.

\begin{proposition}\label{prop:nondegen}
    The reduced-order symplectic matrix $\hat{~J}=~U^\intercal~J~U\in\mathbb{R}^{2m\times 2m}$ is invertible for almost every column-orthonormal reduced basis $~U\in\mathbb{R}^{2M\times 2m}, ~U^\intercal~U=~I$.  
\end{proposition}
\begin{proof}
    Let $N=2M$ and $n=2m$.  Denote by $S$ the subset of the Stiefel manifold $V_n\lr{\mathbb{R}^N} = \{~U\in \mathbb{R}^{N\times n}\,\,|\,\,~U^\intercal~U=~I\}$ of orthonormal $n$-frames in $\mathbb{R}^N$ satisfying $\mathrm{det}\,\hat{~J}=\mathrm{det}\lr{~U^\intercal~J~U}\neq 0$ for all $~U\in S$, i.e. $S=\{~U\in V_n\lr{\mathbb{R}^N}\,|\,\mathrm{det}\,\hat{~J}\neq 0\}$.  This subset is Zariski-open since $\mathrm{det}\,\hat{~J}$ is a polynomial in the entries of $~U$ and nonempty since $~U = \begin{pmatrix}~e_1 & ... & ~e_m & ~e_{M+1} & ... & ~e_{M+m}\end{pmatrix}$ satisfies $\mathrm{det}\,\hat{~J} = \mathrm{det}\,~J_n = 1\neq 0$, where $~e_i$ denotes the $i^{th}$ standard basis vector in $\mathbb{R}^{2M}$.  
    It follows that $S$ is dense in the Stiefel manifold $V_n\lr{\mathbb{R}^N}$, and since polynomials are real-analytic, hence cannot vanish on a set of positive measure without being identically zero, it must be that $S$ has full Lebesgue measure in this space.  Therefore, almost every reduced basis $~U\in V_n\lr{\mathbb{R}^N}$ lies in $S$, as claimed.
\end{proof}

\begin{remark}
    Bases $~U$ for vector subspaces on which the reduced symplectic form $\hat{\omega}$ pulls back to be identically zero, e.g. $~U=\begin{pmatrix}~e_1 & ... & ~e_m\end{pmatrix}$, are called isotropic and can be characterized by the vanishing of the reduced symplectic matrix $\hat{~J}\equiv~0$.  More generally, bases $~U$ for vector subspaces on which the determinant $\mathrm{det}\,\hat{~J}=0$ is zero cannot be symplectic themselves since this implies that $\hat{~J}$ has a nontrivial kernel.  Therefore, these subspaces must be excluded from consideration when building any canonical Hamiltonian ROM.  On the other hand, Proposition~\ref{prop:nondegen} shows that, given a reduced basis $~U$, the reduced symplectic matrix $\hat{~J}$ will be invertible with probability one, which is likely why this technical point is not considered in previous work on Hamiltonian ROMs.
\end{remark}

Notice that the proposed ROM \eqref{eq:newROM} is both energy conserving and symplectic: the reduced symplectic matrix $\hat{~J}$ is obviously constant in $\hat{~x}$, implying closure in the reduced symplectic 2-form $\hat{\omega}\lr{\hat{~v},\hat{~w}}=\hat{~v}\cdot\hat{~J}\hat{~w}$, and the skew symmetry in $\hat{~J}^{-\intercal}=-\hat{~J}^{-1}$ implies that the instantaneous Hamiltonian satisfies
\[\dot{\hat{H}} = \dot{\hat{~x}}\cdot\nabla\hat{H} = \hat{~J}^{-\intercal}\nabla{\hat{H}}\cdot\nabla\hat{H} = -\nabla\hat{H}\cdot\hat{~J}^{-\intercal}\nabla\hat{H} = 0.\]
Remarkably, this variationally consistent ROM is also noncanonical unless $~U^\intercal~J=\hat{~J}~U^\intercal$, in which case $~U^\intercal~J~U=~J_n$ and \eqref{eq:newROM} reduces to the same canonical Hamiltonian ROM \eqref{eq:cotliftROM} produced by a PSD or symplectic lift.  This use of a noncanonical reduced symplectic matrix $\hat{~J}=~U^\intercal~J~U$ resulting from an arbitrary basis $~U$ is similar to the approach in \cite{gong2017structure} leading to the ROM \eqref{eq:symROM}, although now it is apparent that the direct use of the symplectic matrix $\hat{~J}$ in defining the reduced vector field $\dot{\hat{~x}}$ is kinematically inconsistent with the FOM; what is instead needed for consistency is its inverse-transpose $\hat{~J}^{-\intercal}$, which only reduces to $\hat{~J}$ in the special case of a symplectic lift.  The form of the ROM \eqref{eq:newROM} also highlights the remarkable fact that a canonical Hamiltonian ROM is not always appropriate even when reducing a canonical Hamiltonian FOM, meaning that the flexibility $\hat{~J}=~B^\intercal~J_n~B$ of conjugation by an invertible matrix $~B$ afforded in the reduced symplectic matrix may actually increase expressiveness in the resulting Hamiltonian ROM. 


\begin{remark}
    When the FOM \eqref{eq:fom} is equivalently described as $~J^\intercal\dot{~x}=\nabla H(~x)$, again using the nondegeneracy of the symplectic structure, there is an interesting dual interpretation of the ROM \eqref{eq:newROM} as the natural Galerkin ROM with trial/test space given by the image of $~U$ and denoted by $\mathrm{im}\,~U$.  This again displays its variational consistency, since the Hamiltonian ROM \eqref{eq:newROM} is derived directly from a Hamiltonian FOM through a straightforward basis projection.
\end{remark}



\subsection{Intrusive centering}\label{sbsec:int-centering}

At this point, a variationally consistent intrusive Hamiltonian ROM \eqref{eq:newROM} has been presented which is valid for almost any orthonormal reduced basis $~U$ and is guaranteed to conserve the energy level  $\hat{H}\lr{\hat{~x}_0}=H\lr{~P_U~x_0}$ exactly.  Since $H\lr{~P_U~x_0}$ is not necessarily equal to the true energy level $H\lr{~x_0}$ of the corresponding FOM trajectory, it can be beneficial to introduce a well known centering procedure which ensures that this discrepancy is eliminated.  While this centering will be formulated in terms of the initial state $~x_0$, note  that all benefits of the centering discussed here remain valid when $\bar{~x}$ is a linear combination of state snapshots $\{~x_i\}$, say, the mean initial condition along several trajectories.  This can be beneficial when perturbations from the mean are easier to resolve than perturbations from any particular trajectory.


In the intrusive case, the technique of using a centered POD basis for ROM development is well-known and has many benefits.  In \cite{Gunzburger:2007, Barnett:2022, deCastro:2022, deCastro:2023} it is used to enforce inhomogeneous Dirichlet boundary conditions in Galerkin ROMs, while in \cite{Sockwell_THESIS_2019,gong2017structure,gruber2023canonical} it is used for the purpose of preserving linear invariants in Hamiltonian ROMs.  The main idea at work is the inclusion of an affine shift in the basis $~U$ so that it becomes ``centered around zero'' in whatever way is appropriate for a given application.  This is most easily accomplished by choosing $~U$ through the truncated SVD of a centered snapshot matrix $~X-\bar{~X} = ~U~\Sigma~V^\intercal$, where $\bar{~X}$ is a matrix each column of which is a judiciously chosen constant vector $\bar{~x}$.  In the case of Hamiltonian systems, it is advantageous to choose this centering vector $\bar{~x}=~x_0$ to be the initial condition for the Hamiltonian flow, as this ensures that the ROM approximation $\tilde{~x}=~x_0 + ~U\hat{~x}$ satisfies the zero initial condition $\hat{~x}_0=~0$ and maintains the true initial energy value $\hat{H}\lr{\hat{~x}_0}=H\lr{~x_0}$ along the entire ROM trajectory.

Before moving to the discussion of OpInf ROMs, it is worth mentioning another benefit of centered POD important for noncanonical Hamiltonian systems $\dot{~x}=~J\lr{~x}\nabla H\lr{~x}$ where $~J(~x)$ may have a nontrivial kernel.  In \cite{Sockwell_THESIS_2019} it was observed that centering of the reduced basis $~U$ by the initial condition $~x_0$ has the added benefit of preserving linear invariant quantities of a Hamiltonian FOM system, known as linear Casimir invariants (or Casimirs), in a corresponding Hamiltonian ROM.  To see why this is the case, first notice that a linear invariant can be written as $C(~x) = ~c^\intercal~x$ where $~c = \nabla C\in\mathbb{R}^N$ is a constant vector.  It follows that the Casimir
\[C\lr{~x_0 + ~U\hat{~x}} = ~c^\intercal~x_0 + ~c^\intercal~U\hat{~x},\]
will be conserved at the value $~c^\intercal~x_0$ along the ROM trajectory $\hat{~x}$ provided the POD basis $~U$ is constructed so that $~c^\intercal~U\hat{~x}=0$.  Moreover, this will automatically be true when $~U$ is defined through the truncated SVD $~X-~X_0 \approx ~U~\Sigma~V^\intercal$: since each column $~u_i$ of $~U$ takes the form $~u_i=\lr{~X-~X_0}\sigma_i^{-1}~v_i$ in terms of a corresponding singular value $\sigma_i$ and column $~v_i$ of $~V$, it follows that 
\[~c^\intercal~U\hat{~x} = \sum_{i=1}^n \hat{x}^i~c^\intercal~u_i = \sum_{i=1}^n \hat{x}^i~c^\intercal\lr{~X-~X_0}\sigma_i^{-1}~v_i = 0,  \]
where $\hat{x}^i$ denotes the $i^{\mathrm{th}}$ component of the reduced state $\hat{~x}$ and the final equality uses the fact that $~c^\intercal\lr{~X-~X_0}=~0$, since all columns of the snapshot matrix $~X$ lie on the same Casimir-preserving trajectory as $~x_0$.  While not immediately relevant for the ROMs formulated here, this useful fact enables Casimir-preserving schemes which are essential for realistic physics in more complicated Hamiltonian systems, the variationally consistent model reduction of which is the subject of future work.

\begin{remark}
    To improve expressivity of the basis $~U$, it is occasionally useful to include additional snapshots in the matrix $~X$, such as snapshots of $\nabla H$.  When linear Casimir preservation is important, care has to be taken to ensure that these added snapshots still satisfy the conservation condition $~c^\intercal~U\hat{~x}=0$.  One simple way to accomplish this is to project out the component of $\nabla H$ in the direction of $~c$, so that $\nabla^{||}H = ~P_c^\perp\nabla H = \lr{~I-\frac{~c~c^\intercal}{|~c|^2}}\nabla H$ is included in the snapshot matrix instead of $\nabla H$ (c.f. \cite{Sockwell_THESIS_2019}).  This is equivalent to using the natural (i.e. Riemannian) gradient of the Hamiltonian $H$ which is tangent to the linear subspace $\{~x\,|\,~c^\intercal~x=0\}\subset \mathbb{R}^N$.
\end{remark}

\section{Variationally Consistent Hamiltonian Operator Inference}\label{sec:opinf}
Recall that operator inference (OpInf) \cite{peherstorfer2016data} is a useful procedure by which the operators governing a POD ROM are learned directly from snapshot data for the FOM system, ensuring that the resulting ROM is solvable without intrusion into the FOM code.  To apply the proposed variationally consistent Hamiltonian ROM \eqref{eq:newROM} in nonintrusive OpInf settings, it is necessary to understand which challenges with OpInf for Hamiltonian systems have already been addressed and which remain to be overcome.  Here, some previous work on Hamiltonian OpInf is recalled, and an algorithm is presented which enables the use of the variationally consistent approach presented in Section~\ref{sec:method} in cases where only snapshot data are available.






\subsection{OpInf for canonical Hamiltonian systems}\label{sbsec:opinf-old}
Both of the Hamiltonian ROMs \eqref{eq:symROM} and \eqref{eq:newROM} can be effectively treated with OpInf techniques, provided that the nonlinear part of $\nabla H$ is accessible online.  In many examples of interest this is not a large restriction, since the relevant nonlinearities are vectorized over each degree of freedom separately (e.g., polynomials, $\sin(~x)$), and can be easily simulated without understanding of the underlying discretization or access to the FOM code.  Conversely, the linear part of $\nabla H$ is heavily discretization-dependent and generally impossible to simulate without understanding of the full-order scheme or access to its code.  To see how a Hamiltonian OpInf problem is set up in this case, consider writing the Hamiltonian in a way which separates out its quadratic part,
\[H(~x) = \frac{1}{2}~x^\intercal~A~x + f(~x),\]
where $~A\in\mathbb{R}^{N\times N}$ is constant in $~x$ and $f(~x):\mathbb{R}^N\to\mathbb{R}$ is a nonquadratic and likely nonlinear function of the state $~x$.  In this form, the gradient of the Hamiltonian $\nabla H(~x) = ~A~x + \nabla f(~x)$ separates as well, so that if the nonlinear part $\nabla f$ of $\nabla H$ is accessible then the only quantity to be inferred is $~A$.  Accomplishing this inference in a way which is compatible with the least-squares Hamiltonian ROM \eqref{eq:symROM} has already been accomplished in \cite{sharma2022hamiltonian,gruber2023canonical}, so it remains to discuss this technique with the aim of adapting it to the case of the variationally consistent Petrov-Galerkin ROM \eqref{eq:newROM} from Section~\ref{sec:method}.  

Recalling the canonical Hamiltonian operator inference (C-H-OpInf) procedure in \cite{gruber2023canonical}, for an uncentered approximation $\tilde{~x}=~U\hat{~x}$ of the state, the goal is to infer the linear part $\hat{~A} = ~U^\intercal~A~U$ of the reduced Hamiltonian gradient $\nabla \hat{H}\lr{~x} = ~U^\intercal\nabla H(~U\hat{~x})$ under the assumption that $\hat{~A}=\hat{~A}^\intercal$ is symmetric and the nonlinear part $\nabla \hat{f}\lr{\hat{~x}} = ~U^\intercal\nabla f\lr{~U\hat{~x}}$ is either polynomial or can be simulated online.  In this case, denoting the reduced state snapshot matrix by $\hat{~X} = ~U^\intercal~X$, the reduced nonlinearity snapshot matrix by $\hat{\nabla}f\lr{~X} = ~U^\intercal\nabla f\lr{~X}$, and an approximation to the flow map velocity by $D_t\lr{~x}\approx \dot{~\varphi}_t\lr{~x_0} = \dot{~\varphi}_0\lr{~x}$, the problem to be solved is
\[ \bar{~A} = \argmin_{\hat{~A}\in\mathbb{R}^{n\times n}}\,\nn{~U^\intercal D_t\lr{~X} - \hat{~J}\lr{\hat{~A}\hat{~X} + \hat{\nabla} f\lr{~X}}}_F^2 \quad \mathrm{s.t.}\quad \hat{~A}^\intercal=\hat{~A}, \]
which is convex and has a unique global minimum provided $\hat{~X}$ has maximal rank (note that $\mathrm{rank}\,\hat{~J}=n$ by construction).  Here, the velocity snapshot matrix $D_t\lr{~X}$ is an abuse of notation containing true or approximate time derivatives in each column, so that $D_t\lr{~x_k} \approx \dot{~\varphi}_0\lr{~x_k}$ at column $k$.  In fact, letting $\hat{~F}=\hat{\nabla}f\lr{~X}$ denote the reduced nonlinearity snapshots and $\hat{~X}_t = ~U^\intercal D_t\lr{~X}$ denote the approximate velocity field snapshots for brevity, straightforward differentiation from the obvious Lagrangian function shows that this minimization is solvable through the equivalent linear system
\begin{align}\label{eq:symopinf}
    \lr{\hat{~J}^\intercal\hat{~J}\baroplus\hat{~X}\hat{~X}^\intercal}\vect\bar{~A} =\vect\lr{\hat{~J}^\intercal\hat{~X}_t\hat{~X}^\intercal+\hat{~X}\hat{~X}^\intercal_t\hat{~J} - \hat{~J}^\intercal\hat{~J}\hat{~F}\hat{~X}^\intercal - \hat{~X}\hat{~F}^\intercal\hat{~J}^\intercal\hat{~J}},
\end{align}
where $\vect$ denotes the column-wise vectorization operator, $\otimes$ denotes the Kronecker product of matrices, and $~A\baroplus~B = ~A\otimes~B + ~B\otimes ~A$.  More details about the Kronecker product and vectorization can be found in, e.g., \cite{schacke2004kronecker}.  Moreover, note that many FOM systems of interest, including those in Section~\ref{sbsec:linelastic}, automatically compute the exact velocities $\dot{~\varphi}_0(~x)$, which can significantly improve the accuracy of the resulting OpInf ROM (c.f. Theorem~\ref{thm:newop}).



\begin{remark}
    In the case that a symplectic lift basis $~U$ is used, the inference in \eqref{eq:symopinf} reduces to the original H-OpInf procedure of \cite{sharma2022hamiltonian}.
\end{remark}

Once $\bar{~A} \approx \hat{~A}$ has been inferred by solving \eqref{eq:symopinf}, C-H-OpInf simulates the least-squares Hamiltonian ROM \eqref{eq:symROM} online as 
\begin{align*}
    \dot{\hat{~x}} = \hat{~J}\lr{\bar{~A}\hat{~x} + \nabla \hat{f}\lr{\hat{~x}}}, \qquad \hat{~x}_0 = ~U^\intercal~x_0.
\end{align*}
While this minimally intrusive ROM has been shown in \cite{gruber2023canonical} to yield good results in several cases of interest, especially after making the approximation that $\hat{~J}^\intercal\hat{~J}\approx~I$ which is exact for $~J$-equivariant bases, it suffers from the same difficulties as the intrusive ROM \eqref{eq:symROM} which uses the least-squares solution $\hat{~J}=~U^\intercal~J~U$, rendering it ineffective on moderate-to-large-scale linear elastic problems frequently considered in practice (c.f. Section~\ref{sbsec:linelastic}).  Even on problems where the least-squares ROM \eqref{eq:symROM} works well, the OpInf counterpart using the gradient approximation $\bar{~A}\approx\hat{~A}$ often plateaus in accuracy due to nontrivial optimization error \cite{sharma2022hamiltonian,gruber2023canonical}.  Thankfully, the work in \cite{peherstorfer2020sampling} showed that there is a minimally intrusive way to remove this error almost entirely, which will now be discussed alongside the modifications necessary to develop an OpInf ROM for the variationally consistent ROM \eqref{eq:newROM}.

\subsection{Re-projection}\label{sbsec:reproj}
While OpInf is a useful technique for generating non-intrusive reduced models, the original procedure in \cite{peherstorfer2016data} will generally never pre-asymptotically recover the intrusive Galerkin ROM.  This is a straightforward consequence of how the FOM flow map $~\varphi_t(~x)$ interacts with Galerkin projection.  More precisely, notice that when  $\dot{~\varphi}_t(~x_0) = \dot{~x}(t)\coloneqq ~f\lr{~x(t)}$ describes the FOM dynamics, then both the reduced vector field $~U^\intercal~f\lr{~U\hat{~x}}$, corresponding to the (uncentered) intrusive ROM $\dot{\hat{~x}} = ~U^\intercal~f\lr{~U\hat{~x}}$, as well as the $~U$-image of the reduced flow map $~U\hat{~\varphi}_t\lr{\hat{~x}_0} = ~P_U~\varphi_t\lr{~U\hat{~x_0}}$, contain no component orthogonal to the image of the reduced basis $~U$.  Conversely, integrating the FOM \eqref{eq:fom} for any length of time will almost certainly leave the image of $~U$ even when the initial condition is contained there.  This can be seen by letting $~V$ be a basis for the orthogonal complement of $~U$ (so that $~U^\intercal~V=~0$) and expressing the flow at time $t$ with initial condition $~P_U~x_0$ as 
\[~\varphi_t\lr{~P_U~x_0} = ~P_U~\varphi_t\lr{~P_U~x_0} + ~P_V~\varphi_t\lr{~P_U~x_0}.\]
In particular, neither the state $~x=~\varphi_t\lr{~P_U~x_0}$ nor its velocity $\dot{~x} = \dot{~\varphi}_t\lr{~P_U~x_0}$ lie entirely in the image of $~U$ when $~P_V~\varphi_t\lr{~P_U~x_0}$ is nonzero, and this has consequences for standard OpInf.

To see why this is the case, recall that standard OpInf makes use of (approximations to) the velocity data $~U^\intercal D_t\lr{~X} \approx ~U^\intercal\dot{~\varphi}_0\lr{~X} = ~U^\intercal~f\lr{~X}$, each column of which looks like $~U^\intercal~f\lr{~x_k}$ corresponding to the point $t=k\Delta t$ in time.  While this quantity lies squarely in the image of $~U$, it does not represent the velocity field of any reduced flow map.  This can be seen by writing the state $~x$ at any time $t$ as 
\begin{align*}
    ~x=~P_U~x + ~P_V~x\coloneqq ~U\hat{~x} + ~V\hat{~y},
\end{align*}
in terms of coefficient vectors $\hat{~x}=~U^\intercal~x$ and $\hat{~y}=~V^\intercal{~x}$.
It follows that the vector field $~U^\intercal~f\lr{~x}$ admits the expression
\begin{align*}
    ~U^\intercal~f\lr{~x} = ~U^\intercal~f\lr{~U\hat{~x}+~V\hat{~y}} \neq \hat{~f}\lr{\hat{~x}},
\end{align*}
which is not equal to a low-dimensional vector field $\hat{~f}\lr{\hat{~x}}$ for any choice of $\hat{~f}$ due to the presence of $~P_V~x$, the portion of the state $~x$ in the image of $~V$.  This reflects the fact that $~P_U~\varphi_{\Delta t}\lr{~x} \neq ~P_U~\varphi_{\Delta t}\lr{~P_U~x}$ at the level of the FOM flow mapping, since integration in time is not restricted to the image of $~U$, and also means that a correction to the standard OpInf procedure is required to obtain pre-asymptotic recovery of the operators governing the intrusive ROM.


To remedy this issue, a minimally intrusive strategy was given in \cite{peherstorfer2020sampling}, termed ``re-projection'' and viewed through the lens of the Mori-Zwanzig formalism \cite{Parish:2017}.  The idea behind this approach is to collect snapshot data for the FOM vector field $~f$ at the re-projected states $~P_U~X$ instead of (or in addition to) the original states $~X$, in which case the ``non-Markovian'' snapshot matrix $~U^\intercal D_t\lr{~X} \approx ~U^\intercal\dot{~\varphi}_t\lr{~X} = ~U^\intercal~f\lr{~X}$ appearing in the OpInf minimization can be replaced with the ``Markovian'' matrix $~U^\intercal D_t\lr{~P_U~X} \approx ~U^\intercal\dot{~\varphi}_t\lr{~P_U~X}$, where we have abused notation slightly by vectorizing the velocity map $\dot{~\varphi}_t$ column-wise.  Notice that this precisely addresses the issue seen before: instead of drawing snapshots from the projected FOM vector field $~U^\intercal~f\lr{~U\hat{~x}+~V\hat{~y}}$ containing contributions from the image of $~V$, snapshots are drawn from the reduced vector field $~U^\intercal~f\lr{~P_U~x} = ~U^\intercal~f\lr{~U\hat{~x}}$ which depends only on the image of $~U$.  This ensures that the time derivatives used in the OpInf procedure correspond directly to the targeted model form matching the intrusive ROM, allowing for lossless recovery of the intrusive operator independent of the reduced dimension $n$ when these derivatives are exact, i.e. $D_t = \dot{~\varphi}_0$.  The connection to statistical dynamics mentioned in \cite{peherstorfer2020sampling} comes from considering the single-step flow map $~x_{k+1} = ~\varphi_{\Delta t}\lr{~x_k}$, where with similar reasoning to the above it can be seen that the projection $~P_U~\varphi_{\Delta t}\lr{~x_k}$ will generally depend on the entire time history of the state $~x_k$ and therefore cannot be a low-dimensional flow map itself.  The algorithm in \cite{peherstorfer2020sampling} shows that re-projection successfully mitigates the optimization error in OpInf provided the user is able to collect solution snapshots with pre-specified initial conditions.

\begin{remark}
    In the case that the derivative operator $D_t \approx \dot{~\varphi}_0$ is not exact but, e.g., approximated with finite differences, then care must be taken to construct and differentiate a new trajectory using the single-step ROM flow map approximation $~P_U~\varphi_{\Delta t}\lr{~P_U~x} \approx ~U\hat{~\varphi}_{\Delta t}\lr{~U^\intercal~x}$.  Then, the approximate velocity is calculated as $D_t\lr{~x} = \mathcal{FD}_{\Delta t}\lr{~P_U~\varphi_{\Delta t}\lr{~P_U~x}}$ where $\mathcal{FD}_{\Delta t}$ denotes the finite difference approximation with increment $\Delta t$, which converges to $\dot{~\varphi}_0\lr{~x}$ as the increment size $\Delta t\to 0$ approaches zero.  In particular, $D_t\lr{~x}$ is not linear in the state $~x$, and applying the differential operator $\mathcal{FD}_{\Delta t}$ requires samples along the re-projected trajectory defined by $~P_U~\varphi_{\Delta t}~P_U$.
\end{remark}

\subsection{Nonintrusive centering}\label{sbsec:nonint-centering}
It will be desirable to apply centering in the variationally consistent Hamiltonian OpInf ROM for the same reasons as discussed in Section \ref{sbsec:int-centering}.  However, this requires some extension of the ideas discussed there, since it is commonly assumed in OpInf that the approximate state $\tilde{~x}$ is uncentered, and therefore defined by Galerkin projection $\tilde{~x} = ~U\hat{~x}$ directly onto the span of the (uncentered) reduced basis $~U$.  The reason for this is as follows: using a linear FOM $\dot{~x}=~M~x$ where $~M\in\mathbb{R}^{N\times N}$ is a constant matrix operator for illustration, the usual intrusive Galerkin ROM with centered approximation $\tilde{~x}=~x_0+~U\hat{~x}$ is given by 
\[\dot{\hat{~x}} = ~U^\intercal~M~x_0 + ~U^\intercal~M~U\hat{~x} = ~U^\intercal~M~x_0 + \hat{~M}\hat{~x}, \quad \hat{~x}_0 = ~0.\]
In the uncentered case, the reduced matrix operator $\hat{~M} = ~U^\intercal~M~U$ is straightforward to infer without knowledge of $~M$ following the standard procedure in \cite{peherstorfer2016data}.  On the other hand, when the approximate state $\tilde{~x}$ is centered it is not immediately obvious how to nonintrusively access the term $~U^\intercal~M~x_0$ necessary for ROM simulation, which involves a product between the unknown matrix $~M$ and the initial condition $~x_0$.  While it is possible to incorporate this product directly in the governing model form through a learnable constant term, e.g., \cite{mcquarrie2021data} and \cite{Farcas:2023}, this may not yield the desired energy level preservation in the resulting ROM due to the presence of nontrivial optimization error.
However, notice that the problematic term $~U^\intercal~M~x_0 = ~U^\intercal\dot{~x}|_{t=0} \eqqcolon ~U^\intercal~v_0$ can be written in terms of the reduced initial velocity vector $\hat{~v}_0 = ~U^\intercal~v_0$, which is available at least in approximate form. Therefore, as long as the time derivative of the snapshots can be computed or accurately estimated, which is necessary to carry out OpInf in the first place, it is feasible to carry out a centered inference for the reduced operator
\[\bar{~M}= \argmin_{\hat{~M}}\,\nn{~U^\intercal\lr{D_t\lr{~X}-~V_0} - \hat{~M}~U^\intercal\lr{~X-~X_0}}_F^2,\]
where
$~V_0$ is a matrix each column of which is $~v_0$.  
This leads to the centered OpInf POD ROM,
\[\dot{\hat{~x}} = \hat{~v}_0 + \bar{~M}\hat{~x}, \quad \hat{~x}_0 = ~0, \]
which can be simulated without intrusion into the FOM code.

\begin{remark}
While expressing the inaccessible part of the centered ROM in terms of snapshots of the velocity field $\dot{~x}$ is a useful trick which is suitable for the present purposes, it is important to note that this strategy becomes more complicated for nonlinear FOM models or ROMs which utilize nonlinear reduced bases.  
For example, the simplest quadratic model $\dot{~x}=~x^2$ leads to a POD ROM requiring the term $\lr{~x_0+~U\hat{~x}}^2 = ~x_0^2 + 2\,\mathrm{Diag}\lr{~x_0}~U\hat{~x} + \mathrm{diag}_{1,3}\lr{~U\otimes~U}\lr{~x\otimes ~x}$, where $\mathrm{Diag}(~x_0)$ is a diagonal matrix with entries $~x_0$ and the order-3 tensor $\mathrm{diag}_{1,3}\lr{~U\otimes~U}^i_{jk} = u^i_ju^i_k$.  Applying the same trick in this case yields $~v_0 = ~x_0^2$, which does not recover the entire contribution of $~x_0$.  On the other hand, given $~v_0$ as before, standard OpInf \cite{mcquarrie2021data} can be carried out for the inaccessible operators $2\,\mathrm{Diag}\lr{~x_0}~U$ and $\mathrm{diag}_{1,3}\lr{~U\otimes~U}$ simultaneously, perhaps with some optimization error incurred. It is straightforward to check that this general recipe will work for all polynomial nonlinearities, and extending centering to nonintrusive ROMs with more complex nonlinearities is an interesting avenue for future work.
\end{remark}

\subsection{Variationally consistent Hamiltonian OpInf}

It is now pertinent to discuss how OpInf can be effectively applied to the variationally consistent ROM \eqref{eq:newROM}, with or without a centered state approximation.  Recalling the splitting $H(~x) = \frac{1}{2}~x^\intercal~A~x + f(~x)$ of the Hamiltonian from Section~\ref{sbsec:opinf-old}, the standard OpInf problem associated to the uncentered ROM \eqref{eq:newROM} is given as
\[\bar{~A} = \argmin_{\hat{~A}\in\mathbb{R}^{n\times n}}\,\nn{\lr{~J~U}^\intercal D_t\lr{~X} - \hat{~A}\hat{~X}-\hat{\nabla}f\lr{~X}}_F^2 \quad\mathrm{s.t.}\quad \hat{~A}^\intercal=\hat{~A}.\]
This problem has a unique minimum under the same rank constraints as in Section~\ref{sbsec:opinf-old} while requiring less matrix multiplications with the reduced-order Poisson matrix $\hat{~J} = ~U^\intercal~J~U$, which is an advantage of the Petrov-Galerkin formulation. On the other hand, it is clear from Section~\ref{sbsec:reproj} that the learned operator $\bar{~A}$ will not generally match the intrusive ROM operator $\hat{~A}$, due to closure error introduced by the velocity approximation $D_t\lr{~X}$ and the reduced nonlinearity approximation $\hat{\nabla}f\lr{~X}$.  When snapshots of the projected state $~P_U~x$ can be collected, this error can be eliminated by carrying out the re-projected inference,
\begin{align}\label{eq:reprojnewop}
    \bar{~A} = \argmin_{\hat{~A}\in\mathbb{R}^{n\times n}}\,\nn{\lr{~J~U}^\intercal D_t\lr{~P_U~X} - \hat{~A}\hat{~X}-\hat{\nabla}f\lr{~P_U~X}}_F^2 \quad\mathrm{s.t.}\quad \hat{~A}^\intercal=\hat{~A},
\end{align}
where the approximate re-projected velocity field $D_t\lr{~P_U~X}\approx \dot{~\varphi}_0\lr{~P_U~X}$ is now used.  Letting $~X_t=D_t\lr{~P_U~X}$ denote the re-projected velocity field and $\hat{~F} = \hat{\nabla}f\lr{~P_U~X}$ denote the reduced, re-projected gradient of the nonlinear term (warning: these are different from the $~X_t,\hat{~F}$ defined before), it can be shown using \cite[Theorem 3.1]{gruber2023canonical} that the minimization \eqref{eq:reprojnewop} is equivalent to the linear system
\begin{align}\label{eq:linsys-newop}
    \lr{~I\baroplus\hat{~X}\hat{~X}^\intercal}\vect\bar{~A} = \vect\lr{~U^\intercal~J^\intercal~X_t\hat{~X}^\intercal + \hat{~X}~X_t^\intercal~J~U - \hat{~F}\hat{~X}^\intercal - \hat{~X}\hat{~F}^\intercal}.
\end{align}
Importantly, this inference yields a learned operator $\bar{~A}\approx\hat{~A}$ which matches the intrusive ROM operator $\hat{~A}=~U^\intercal~A~U$ up to numerical error from the linear solve.  Of course, once $\bar{~A}$ has been determined from the above, a non-intrusive version of the variationally consistent ROM \eqref{eq:newROM} can be simulated as
\begin{align}\label{eq:opinfROM-uncentered}
    \hat{~J}^\intercal\dot{\hat{~x}} = \bar{~A}\hat{~x} + \nabla\hat{f}\lr{\hat{~x}}, \quad \hat{~x}_0 = ~U^\intercal~x_0.
\end{align}


On the other hand, it is often desirable to include a shift $\tilde{~x} = ~x_0 + ~U\hat{~x}$ in the ROM state approximation for the reasons described in Section~\ref{sbsec:int-centering}, corresponding to the centered Hamiltonian ROM analogous to \eqref{eq:newROM}:
\begin{align}\label{eq:newROM-centered}
    \hat{~J}^\intercal\dot{\hat{~x}} = ~U^\intercal\nabla H\lr{~x_0+~U\hat{~x}} = \nabla\hat{H}\lr{\hat{~x}}, \quad \hat{~x}_0 = ~0.
\end{align}
When $\nabla H = ~A$ is linear, such as in the solid mechanical examples of Section~\ref{sec:numerics}, or when the nonlinear part $\nabla f$ can be accessed intrusively, this centering carries over to OpInf.  Applying the novel procedure described in Section~\ref{sbsec:nonint-centering}, it follows from \eqref{eq:newROM-centered} that 
\begin{align*}
    \hat{~J}^\intercal\dot{\hat{~x}} &= ~U^\intercal~A\lr{~x_0+~U\hat{~x}} + ~U^\intercal\nabla f\lr{~x_0+~U\hat{~x}} \\
    &= \lr{~J~U}^\intercal\lr{~v_0-~J\nabla f\lr{~x_0}} + \hat{~A}\hat{~x} + ~U^\intercal\nabla f\lr{~x_0+~U\hat{~x}},
\end{align*}
where $~v_0 = \dot{~x}|_{t=0}$ denotes the FOM initial velocity.  Therefore, to simulate the centered nonintrusive ROM corresponding to \eqref{eq:newROM-centered},
\begin{equation}\label{eq:opinfROM-centered}
    \hat{~J}^\intercal\dot{\hat{~x}} = \lr{~J~U}^\intercal\lr{~v_0-~J\nabla f\lr{~x_0}} + \bar{~A}\hat{~x} + \nabla\hat{f}\lr{\hat{~x}}, \quad \hat{~x}_0=~0,
\end{equation}
it is enough to solve the re-projected OpInf problem 
\begin{align}\label{eq:reprojnewop-centered}
    \bar{~A} = \argmin_{\hat{~A}\in\mathbb{R}^{n\times n}}\,\nn{~C-\hat{~A}~U^\intercal\lr{~X-~X_0}-~U^\intercal\nabla f\lr{~X_0+~P_U\lr{~X-~X_0}}}_F \quad\mathrm{s.t.}\quad \hat{~A}^\intercal=\hat{~A},
\end{align}
where the matrix $~C = \lr{~J~U}^\intercal D_t\lr{~X_0+~P_U\lr{~X-~X_0}} - D_t\lr{~X_0} + ~J\nabla f\lr{~X_0}$ and each column of the matrix $~X_0$ is the initial condition $~x_0$.  There are a couple of things to point out here.  First, re-projection involves both the state $~x$ and the initial condition $~x_0$ in this case, since if $\mathrm{im}\,~V = \lr{\mathrm{im}\,~U}^\perp$ then $~x = ~x_0 + ~P_U\lr{~x-~x_0} + ~P_V\lr{~x-~x_0}$ is the appropriate decomposition in terms of the subspaces $\mathrm{span}\,~x_0,\mathrm{im}\,~U$ and $\mathrm{im}\,~V$.  Secondly, when the approximate time derivative operator $D_t$ is linear, such as in the case that $D_t = \dot{~\varphi}_t = ~J~A$ is exact, the term $\lr{~J~U}^\intercal D_t\lr{~X_0+~P_U\lr{~X-~X_0}}-D_t\lr{~X_0} = \lr{~J~U}^\intercal D_t\lr{~P_U\lr{~X-~X_0}}$ conveniently simplifies.  In this case, letting $~X_t = D_t\lr{~P_U\lr{~X-~X_0}} +~J\nabla f\lr{~X_0}$ and $\hat{~X}=~U^\intercal\lr{~X-~X_0}$, a reduced operator $\bar{~A}\approx\hat{~A}$ which approximates the intrusive operator up to numerical error is given by the solution to \eqref{eq:linsys-newop} with $\hat{~F}=~0$.

This gives all information necessary for carrying out analysis and experiments using the proposed Hamiltonian ROM in both intrusive and nonintrusive settings.  For ease of use, the information presented here is summarized in Algorithm~\ref{alg:newOpInf}.  Note in particular that all proposed variationally consistent Hamiltonian OpInf ROMs are efficient online, incurring exactly the same online cost as OpInf ROMs which are structure-agnostic.


\begin{algorithm}[htb]
\caption{Variationally Consistent Hamiltonian Operator Inference (VC-H-OpInf)}\label{alg:newOpInf}
\begin{algorithmic}[1]
\Require Snapshots $~X^{N\times \lr{n_s+1}}$ of FOM solution corresponding to times $\mathcal{T}=\{t_0, ..., t_{n_s}\}$; intrusive access to nonlinear term $\nabla f:\mathbb{R}^N\to\mathbb{R}^N$ in Hamiltonian gradient $\nabla H(~x) = ~A~x + \nabla f(~x)$; reduced dimension $n>0$.
\Ensure Reduced operator $\bar{~A}\in\mathbb{R}^{n\times n}$ satisfying $\bar{~A}^\intercal=\bar{~A}$ which approximates the linear term $\hat{~A}=~U^\intercal~A~U$ in the reduced Hamiltonian $\hat{H}\lr{\hat{~x}} = \frac{1}{2}\hat{~x}^\intercal\hat{~A}\hat{~x} + \hat{f}\lr{\hat{~x}}$.
\State Build a reduced basis $~U\in V_n\lr{\mathbb{R}^N}$ from snapshot data $~X$.
\If{centering is employed}
    \If{re-projection is employed}
        \State \multiline{%
            Collect re-projected snapshots $~x_0 = ~x_0$, $~x_k = ~P_U~\varphi_{\Delta t}\lr{~x_0 + ~P_U\lr{~x_{k-1}-~x_0}}$ column-wise in matrix $~X$.  \\\Comment{Skip this if $\dot{~\varphi}_0$ is available.}}
        \State Set reduced state snapshots $\hat{~X}=~U^\intercal\lr{~X-~X_0}$.
        \State Collect re-projected velocity snapshots $D_t\lr{~X_0+~P_U\lr{~X-~X_0}} \approx \dot{~\varphi}_0\lr{~X_0+~P_U\lr{~X-~X_0}}$.
        \State Set derivative matrix $~X_t = D_t\lr{~X_0+~P_U\lr{~X-~X_0}} - D_t\lr{~X_0} + ~J\nabla f\lr{~X_0}$.
        \State Collect re-projected reduced nonlinearity snapshots $\hat{~F} = ~U^\intercal\nabla f\lr{~X_0+~U\lr{~X-~X_0}}$. 
    \Else
        \State Set reduced state snapshots $\hat{~X}=~U^\intercal\lr{~X-~X_0}$.
        \State Collect velocity snapshots $D_t\lr{~X} \approx \dot{~\varphi}_0\lr{~X}$.  
        \State Set derivative matrix $~X_t = D_t\lr{~X}-D_t\lr{~X_0}+~J\nabla f\lr{~X_0}$.
        \State Collect reduced nonlinearity snapshots $\hat{~F} = ~U^\intercal\nabla f\lr{~X}$.
    \EndIf
    \State Solve \eqref{eq:linsys-newop} for the reduced operator $\bar{~A}$ used in the nonintrusive centered ROM \eqref{eq:opinfROM-centered}.
\Else
    \If{re-projection is employed}
        \State \multiline{Collect re-projected snapshots $~x_0 = ~P_U~x_0$, $~x_k = ~P_U~\varphi_{\Delta t}\lr{~P_U~x_{k-1}}$ column-wise in matrix $~X$.  \\\Comment{Skip this if $\dot{~\varphi}_0$ is available.}}
        \State Set reduced state snapshots $\hat{~X} = ~U^\intercal~X$.
        \State Collect re-projected velocity snapshots $~X_t = D_t\lr{~P_U~X} \approx \dot{~\varphi_0}\lr{~P_U~X}$.
        \State Collect re-projected reduced nonlinearity snapshots $\hat{~F} = ~U^\intercal\nabla f\lr{~P_U~X}$.
    \Else
        \State Set reduced state snapshots $\hat{~X} = ~U^\intercal~X$.
        \State Collect velocity snapshots $~X_t = D_t\lr{~X} \approx \dot{~\varphi}_0\lr{~X}$.
        \State Collect reduced nonlinearity snapshots $\hat{~F} = ~U^\intercal\nabla f\lr{~X}$.
    \EndIf
    \State Solve \eqref{eq:linsys-newop} for the reduced operator $\bar{~A}$ used in the nonintrusive uncentered ROM \eqref{eq:opinfROM-uncentered}.
\EndIf
\end{algorithmic}
\end{algorithm}


\section{Analysis} \label{sec:analysis}


The proposed variationally consistent Hamiltonian ROM has now been established in four forms: uncentered and intrusive \eqref{eq:newROM}, centered and intrusive \eqref{eq:newROM-centered}, uncentered and nonintrusive \eqref{eq:opinfROM-uncentered}, and centered and nonintrusive \eqref{eq:opinfROM-centered}.  It will now be shown that state error is bounded in each of these cases by easily interpretable quantities coming from the reduced basis projection and the deviation of these ROMs from being canonically Hamiltonian.
To that end, denote the $L^2$ norm of the solution $~x(t)$ on the interval $[0,T]\subset\mathbb{R}$ by 
\[\left\|~x\right\|^2 = \int_0^T \nn{~x(t)}^2\,dt,\]
and denote projection onto the orthogonal complement of the span of $~U$ by $~P_U^\perp = ~I-~U~U^\intercal$.
To develop an error estimate for the intrusive ROMs \eqref{eq:newROM} and \eqref{eq:newROM-centered}, first recall the following classical result which bounds projection error in terms of the eigenvalues of the covariance matrix.

\begin{proposition}\label{prop:POD}
    Let $~U\in\mathbb{R}^{N\times n}$ be the rank-$n$ POD basis generated through snapshots of $~x(t)$ on the interval $[0,T]\subset\mathbb{R}$.  Then, the POD projection error satisfies
    \[\left\|~P_U^\perp~x\right\|^2 = \left\|~x-~P_U~x\right\|^2 = \sum_{i>n} \lambda_i,\]
    where $\lambda_i$ are the eigenvalues of the snapshot covariance matrix $~C\lr{~x} = \int_0^T ~x(t)~x(t)^\intercal dt$.  In particular, if $~U$ is instead built from centered snapshots $~w(t)=~x(t)-\bar{~x}$ then the projection error satisfies
    \[\left\|~P_U^\perp~w\right\|^2 = \left\|~P_U^\perp\lr{~x-\bar{~x}}\right\|^2 = \sum_{i>n} \lambda_i,\]
    where $\lambda_i$ are now the eigenvalues of the centered covariance matrix $~C(~x) = \int_0^T \lr{~x-\bar{~x}}\lr{~x-\bar{~x}}^\intercal dt$.
\end{proposition}

Recall additionally that a function $~f:\mathbb{R}^n\to\mathbb{R}^p$ is Lipschitz continuous provided there exists a finite constant $L_{~f}>0$ such that for any $~x,~y\in\mathbb{R}^n$,
\[\nn{~f(~x)-~f(~y)} \leq L_{~f}\nn{~x-~y},\]
where the norms are Euclidean and defined in the relevant spaces.  Letting $\lambda_{\mathrm{max}}(\cdot)$ denote selection of the maximum eigenvalue, the following error estimate is now proven for the variationally consistent Hamiltonian ROM in the intrusive setting.

\begin{theorem}\label{thm:newrom}
    Let $~x(t)$ denote the solution to the FOM \eqref{eq:fom} on the interval $[0,T]$
    and let $\hat{~x}(t)$ denote the solution to the variationally consistent Hamiltonian ROM \eqref{eq:newROM-centered}.  If the Hamiltonian gradient $\nabla H$ is Lipschitz continuous on the $~x$-image of $[0,T]$, then the approximation error in the reproductive ROM satisfies
    \[\left\|~x - \lr{~x_0 + ~U\hat{~x}}\right\| \leq c\sqrt{\sum_{i>n} \lambda_i}
    + \bar{c}\,\left\|\nabla H(~x)\right\|\sqrt{\lambda_{\mathrm{max}}\lr{\hat{~J}^{-\intercal}\hat{~J}^{-1}-~I}}, \]
    for constants $c,\bar{c}$ depending on $T, ~U, ~J, \nabla H$ and where $\lambda_i$ are the eigenvalues of the centered covariance matrix.  In particular, the ROM error is bounded and tends to zero as the reduced dimension $n$ tends to the full dimension $N$.
\end{theorem}
\begin{proof}
    First, suppose $~x_0=~0$, so that the ROM approximation is uncentered. Adding and subtracting the projection $~P_U~x$, it follows that
    \[~x - ~U\hat{~x} = \lr{~x-~P_U~x} + \lr{~P_U~x - ~U\hat{~x}} \eqqcolon ~P_U^\perp~x + ~y,\]
    so that the integrand of the ROM error decomposes into two distinct terms.  Moreover, by adding and subtracting additional terms it follows that the time derivative of $~y$ can be expressed as
    \begin{align*}
        \dot{~y} &= ~P_U\dot{~x} - ~U\dot{\hat{~x}} = ~P_U~J\nabla H(~x) - ~U\hat{~J}^{-\intercal}~U^\intercal\nabla H\lr{~U\hat{~x}} \\
        &= ~U\lr{~U^\intercal~J - \hat{~J}^{-\intercal}~U^\intercal}\nabla H(~x) + ~U\hat{~J}^{-\intercal}~U^\intercal\lr{\nabla H(~x)-\nabla H\lr{~P_U~x}} \\
        &\quad+ ~U\hat{~J}^{-\intercal}~U^\intercal\lr{\nabla H\lr{~P_U~x}-\nabla H\lr{~U\hat{~x}}}.
    \end{align*}
    Noting that the spectral norm $\nn{~A}^2$ = $\lambda_{\mathrm{max}}\lr{~A^\intercal~A} = \lambda_{\mathrm{max}}\lr{~A~A^\intercal}$ is invariant under unitary transformations then yields the norm bound $\nn{~U\hat{~J}^{-\intercal}~U^\intercal} = \nn{\hat{~J}^{-1}}$ along with the calculation
    \begin{align*}
        \nn{~U\lr{~U^\intercal~J - \hat{~J}^{-\intercal}~U^\intercal}}^2 &= \nn{~U^\intercal~J - \hat{~J}^{-\intercal}~U^\intercal}^2 \\
        &= \lambda_{\mathrm{max}}\lr{~U^\intercal~J~J^\intercal~U-\hat{~J}^{-\intercal}\hat{~J}^\intercal-\hat{~J}\hat{~J}^{-1}+\hat{~J}^{-\intercal}\hat{~J}^{-1}} \\
        &= \lambda_{\mathrm{max}}\lr{\hat{~J}^{-\intercal}\hat{~J}^{-1}-~I}.
    \end{align*} 
    Since the gradient $\nabla H$ is assumed Lipschitz continuous on the $~x$-image of the compact set $[0,T]$, hence bounded by assumption, this implies that the norm of $\dot{~y}$ can be estimated as 
    \begin{align*}
        \nn{\dot{~y}} &\leq \nn{~U\lr{~U^\intercal~J - \hat{~J}^{-\intercal}~U^\intercal}}\nn{\nabla H\lr{~x}} + \nn{~U\hat{~J}^{-\intercal}~U^\intercal}L_{\nabla H}\lr{\nn{~P_U^\perp~x} + \nn{~y}} \\
        &= \sqrt{\lambda_{\mathrm{max}}\lr{\hat{~J}^{-\intercal}\hat{~J}^{-1}-~I}}\,\nn{\nabla H(~x)} + \nn{\hat{~J}^{-1}}L_{\nabla H}\lr{\nn{~P_U^\perp~x} + \nn{~y}}.
    \end{align*}
    Now, using that the time derivative of the norm $\partial_t\nn{~y} = \lr{2\nn{~y}}^{-1}\IP{\dot{~y}}{~y}\leq\nn{\dot{~y}}$ is less than the norm of the time derivative, and defining the constant $c_2 = \nn{\hat{~J}^{-1}}L_{\nabla H}$ leads to a first-order differential equation for $\nn{~y}$,
    \begin{align*}
        \partial_t\nn{~y} - c_2\nn{~y} \leq \nn{\nabla H\lr{~x}}\sqrt{\lambda_{\mathrm{max}}\lr{\hat{~J}^{-\intercal}\hat{~J}^{-1}-~I}} + c_2\nn{~P_U^\perp~x}.
    \end{align*}
    Therefore, multiplying by the integrating factor $e^{-c_2t}$ and applying Gronwall's inequality along with $~y(0)=~0$ yields
    \begin{align*}
        \nn{~y} &\leq \int_0^t e^{c_2\lr{t-s}}\lr{\nn{\nabla H\lr{~x}}\sqrt{\lambda_{\mathrm{max}}\lr{\hat{~J}^{-\intercal}\hat{~J}^{-1}-~I}} + c_2\nn{~P_U^\perp~x}}\, ds \\
        &\leq \int_0^T e^{c_2\lr{T-s}}\lr{\nn{\nabla H\lr{~x}}\sqrt{\lambda_{\mathrm{max}}\lr{\hat{~J}^{-\intercal}\hat{~J}^{-1}-~I}} + c_2\nn{~P_U^\perp~x}}\, ds \\
        &\leq C\, \left\| \nn{\nabla H\lr{~x}}\sqrt{\lambda_{\mathrm{max}}\lr{\hat{~J}^{-\intercal}\hat{~J}^{-1}-~I}} + c_2\nn{~P_U^\perp~x}\right\|, 
    \end{align*}
    where $C = \left\|e^{2c_2\lr{T-t}}\right\| = \sqrt{\lr{e^{2c_2T}-1}/2c_2}$ and the last step applies the Cauchy-Schwarz inequality.  It follows from this and the Minkowski inequality that  
    \begin{align*}
        \left\|~y\right\| &\leq C\sqrt{T}\,\left\|\nn{\nabla H(~x)}\sqrt{\lambda_{\mathrm{max}}\lr{\hat{~J}^{-\intercal}\hat{~J}^{-1}-~I}} + c_2\nn{~P_U^\perp~x}\right\| \\
        &\leq C\sqrt{T}\,\left\|\nabla H(~x)\sqrt{\lambda_{\mathrm{max}}\lr{\hat{~J}^{-\intercal}\hat{~J}^{-1}-~I}}\right\| + Cc_2\sqrt{T}\,\left\|~P_U^\perp~x\right\|,
    \end{align*}
    and therefore the state error is bounded by
    \begin{align*}
        \left\|~x-~U\hat{~x}\right\| \leq \left\|~P_U^\perp~x\right\| + \left\|~y\right\| \leq \lr{1+Cc_2\sqrt{T}}\left\|~P_U^\perp~x\right\| + C\sqrt{T}\,\left\|\nabla H(~x)\right\|\sqrt{\lambda_{\mathrm{max}}\lr{\hat{~J}^{-\intercal}\hat{~J}^{-1}-~I}},
    \end{align*}
    yielding the conclusion when $~x_0=~0$.  Now, given the ROM approximation $\tilde{~x} = ~x_0+~U\hat{~x}$ with nonzero centering $~x_0\neq~0$, then $~w = ~x-~x_0$ satisfies $\dot{~w} = \dot{~x}$ and the ROM error integrand becomes
\[~x-\lr{~x_0+~U\hat{~x}} = \lr{~w - ~P_U~w} + \lr{~P_U~w - ~U\hat{~x}} \eqqcolon ~P_U^\perp~w + ~z. \]
Since the projection error $\left\|~P_U^\perp~w\right\|$ can be bounded by Proposition~\ref{prop:POD}, $~w_0=~0$, and $\dot{~w}=\dot{~x}$, the proof of this case becomes identical to the proven case $~x_0=~0$ with no centering. 
\end{proof}

Theorem~\ref{thm:newrom} shows that the error in the variationally consistent ROM approximation can be separated into two distinct terms  coming from the information lost during POD projection and the failure of the reduced symplectic matrix inverse $\hat{~J}^{-1}$ to square to minus the identity, respectively.  Notice that this second contribution is zero when using a symplectic lift or PSD basis, since the equivariance condition on the reduced symplectic matrix implies that $\hat{~J} = ~J_n$ and the resulting ROM is canonically Hamiltonian.  In this case, state error in the ROM depends only on the POD projection error as consistent with standard Galerkin theory, illustrating why symplectic lift bases tend to lead to the best performance when they are expressive enough to reliably capture FOM state dynamics.  Unfortunately, this is often not the case in realistic simulations, where the usual POD basis formed via the SVD of state snapshots outperforms known algorithms for symplectic lift bases by several orders of magnitude. In these instances, the balance of both terms in Theorem~\ref{thm:newrom} is important, and it can happen that only \eqref{eq:newROM} leads to usable results.  Some examples to this effect can be found in Section~\ref{sbsec:linelastic}.


\begin{remark}
    It is remarkable that every Hamiltonian ROM which does not rely on a specially designed basis introduces extra contributions to the ROM error in order to preserve Hamiltonian structure at the reduced level.  In the case of the least-squares ROM \eqref{eq:symROM}, this term is essentially the projection $~P_U^\perp\nabla H$ of the gradient into $\lr{\mathrm{im}\,~U}^\perp$ (c.f. \cite[Theorem 3.1]{gong2017structure}), while for the proposed variationally consistent ROM \eqref{eq:newROM} it is the failure of the reduced  symplectic matrix inverse $\hat{~J}^{-1}$ to be anti-involutive.  
\end{remark}

It is useful to verify that a similar error bound holds for the nonintrusive ROM \eqref{eq:reprojnewop-centered}. In fact, the same technique can be applied provided care is taken with  other sources of potential error.  In particular, these are error in the time derivative approximation $D_t\lr{~X}\approx \dot{~\varphi}_0\lr{~X}$ and optimization error which arises from solving the relevant OpInf problem.  For each state $~x$, these quantities are defined as
\begin{align*}
    \varepsilon_{\Delta t}(~x) = \nn{\dot{~x}-D_t\lr{~x}}, \qquad \varepsilon_A(~x) = \nn{\lr{~J~U}^\intercal D_t\lr{~x}-~U^\intercal\nabla\bar{H}\lr{~x}},
\end{align*}
where $\bar{H}(~x) = \frac{1}{2}~x^\intercal~U\bar{~A}~U^\intercal~x + f\lr{~x}$ denotes the approximate Hamiltonian defined by the OpInf solution $\bar{~A}\approx\hat{~A}$.  In particular, note that $\varepsilon_A\lr{~P_U~x}$ is the relevant notion of approximation error when re-projection is employed, and that this quantity vanishes when the time derivative approximation $D_t\approx \dot{~\varphi}_0$ is exact.  With these notions in place, there is the following analogue of Theorem~\ref{thm:newrom}.

\begin{theorem}\label{thm:newop}
    Let $~x(t)$ denote the solution to the FOM \eqref{eq:fom} on the interval $[0,T]$
    and let $\hat{~x}(t)$ denote the solution to the variationally consistent OpInf ROM \eqref{eq:opinfROM-centered}.  Under the same assumptions as Theorem~\ref{thm:newrom} and provided the time derivative approximation $D_t \approx \dot{~\varphi}_t$ is Lipschitz continuous and the operator error $\nn{~A-~U\bar{~A}~U^\intercal}$ is finite, the approximation error in the reproductive ROM satisfies 
    \[\left\|~x - \lr{~x_0 + ~U\hat{~x}}\right\| \leq c\sqrt{\sum_{i>n} \lambda_i}
    + \bar{c}\,\left\|D_t\lr{~x}\right\|\sqrt{\lambda_{\mathrm{max}}\lr{\hat{~J}^{-\intercal}\hat{~J}^{-1}-~I}} + \tilde{c}\,\lr{\left\|\varepsilon_{\Delta t}\lr{~x}\right\|+\left\|\varepsilon_A\lr{~P_U~x}\right\|}, \]
    for constants $c,\bar{c}, \tilde{c}$ depending only on $T, ~U, ~J, \nabla H$ and where $\lambda_i$ are the eigenvalues of the centered covariance matrix.  Moreover, when the time derivative operator $D_t = \dot{~\varphi}_t$ is exact, this reduces to precisely the bound in Theorem~\ref{thm:newrom}.
\end{theorem}
\begin{proof}
     First, note that Lipschitz continuity in the Hamiltonian gradient $\nabla H$ implies the same in the approximate Hamiltonian gradient $\nabla\bar{H}$, since by adding and subtracting $~A\lr{~x-~y}$ and applying the triangle inequality it follows that
     \begin{align*}
         \nn{\nabla\bar{H}\lr{~x}-\nabla\bar{H}\lr{~y}} &= \lr{~U\bar{~A}~U^\intercal\lr{~x-~y} + \nabla f(~x) - \nabla f(~y)} \\
         &\leq \lr{\nn{~A-~U\bar{~A}~U^\intercal} + L_{\nabla H}}\nn{~x-~y},
     \end{align*}
     showing that $\nabla\bar{H}$ is Lipschitz continuous with (finite) constant $L_{\nabla\bar{H}} = L_{\nabla H} + \nn{~A-~U\bar{~A}~U^\intercal}$.  Now, suppose $~x_0=~0$, so that the ROM approximation is uncentered. Adding and subtracting the projection $~P_U~x$ as before, it follows that
    \[~x - ~U\hat{~x} = \lr{~x-~P_U~x} + \lr{~P_U~x - ~U\hat{~x}} \eqqcolon ~P_U^\perp~x + ~y,\]
    so that the error decomposes into two distinct terms.  Adding and subtracting more terms, it follows that the time derivative of $~y$ can be expressed as
    \begin{align*}
        \dot{~y} &= ~P_U\dot{~x} - ~U\dot{\hat{~x}} = ~P_U\lr{\dot{~x}-D_t(~x)} + ~P_U\lr{~I-~U\hat{~J}^{-\intercal}~U^\intercal~J^\intercal}D_t(~x) + ~U\hat{~J}^{-\intercal}~U^\intercal~J^\intercal\lr{D_t(~x)-D_t\lr{~P_U~x}} \\
        &\quad\qquad\qquad\qquad+ ~U\hat{~J}^{-\intercal}\lr{\lr{~J~U}^\intercal D_t\lr{~P_U~x} - ~U^\intercal\nabla\bar{H}\lr{~P_U~x}} + ~U\hat{~J}^{-\intercal}~U^\intercal\lr{\nabla \bar{H}\lr{~P_U~x}-\nabla \bar{H}\lr{~U\hat{~x}}}.
    \end{align*}
    Using again the invariance of the spectral norm under unitary transformations along with the calculation,
    \begin{align*}
        \nn{~P_U-~U\hat{~J}^{-\intercal}~U^\intercal~J^\intercal} &= \nn{~U\lr{~U^\intercal~J-\hat{~J}^{-\intercal}~U^\intercal}~J^\intercal} \\
        &= \nn{~U^\intercal~J-\hat{~J}^{-\intercal}~U^\intercal} = \lambda_{\mathrm{max}}\lr{\hat{~J}^{-\intercal}\hat{~J}^{-1}-~I}^{1/2},
    \end{align*}
    leads to the following norm estimate of $\dot{~y}$:
    \begin{align*}
        \nn{\dot{~y}} &\leq \nn{\dot{~x}-D_t(~x)} + \lambda_{\mathrm{max}}\lr{\hat{~J}^{-\intercal}\hat{~J}^{-1}-~I}^{1/2} \nn{D_t(~x)} + \nn{\hat{~J}^{-1}}\nn{D_t(~x)-D_t\lr{~P_U~x}} \\
        &\qquad+ \nn{\hat{{~J}}^{-1}}\nn{\lr{~J~U}^\intercal D_t\lr{~P_U~x} - ~U^\intercal\nabla \bar{H}\lr{~P_U~x}} + \nn{\hat{~J}^{-1}}\nn{\nabla \bar{H}\lr{~P_U~x}-\nabla \bar{H}\lr{~U\hat{~x}}} \\
        &\leq \varepsilon_{\Delta t}(~x) + \lambda_{\mathrm{max}}\lr{\hat{~J}^{-\intercal}\hat{~J}^{-1}-~I}^{1/2} \nn{D_t(~x)} + \nn{\hat{~J}^{-1}} L_{D_t}\nn{~P_U^\perp~x} + \nn{\hat{~J}^{-1}}\varepsilon_A\lr{~P_U~x} + \nn{\hat{~J}^{-1}} L_{\nabla\bar{H}}\nn{~y}.
    \end{align*}
    Defining constants  $c_0=\nn{\hat{~J}^{-1}}$, $c_2=\nn{\hat{~J}^{-1}}L_{\nabla \bar{H}}$, and $c_3=\nn{\hat{~J}^{-1}}L_{D_t}$, and applying $\partial_t\nn{~y} \leq \nn{\dot{~y}}$ again leads to a first-order differential equation, 
    \begin{align*}
        \partial_t\nn{~y} - c_2\nn{~y}\leq \varepsilon_{\Delta t}(~x) + c_0\varepsilon_A\lr{~P_U~x} + \nn{D_t(~x)}\sqrt{\lambda_{\mathrm{max}}\lr{\hat{~J}^{-\intercal}\hat{~J}^{-1}-~I}} + c_3\nn{~P_U^\perp~x},
    \end{align*}
    and therefore $\partial_t\lr{e^{-c_2t}\nn{~y}}$ can be integrated in time and bounded in terms of the constant $C=\sqrt{\lr{e^{2c_2T}-1}/2c_2}$ to yield the pointwise-in-time bound
    \begin{align*}
        \nn{~y} &\leq C\lr{\left\|\varepsilon_{\Delta t}\lr{~x}\right\|+c_0\left\|\varepsilon_A\lr{~P_U~x}\right\|} + Cc_3\left\|~P_U^\perp~x\right\| + C\left\|D_t(~x)\right\|\sqrt{\lambda_{\mathrm{max}}\lr{\hat{~J}^{-\intercal}\hat{~J}^{-1}-~I}}.
    \end{align*}
    Defining the constant $c_1 = \mathrm{max}\{1, c_0\}$ and integrating again in time then yields the desired estimate
    \begin{align*}
        \left\|~x-~U\hat{~x}\right\| &\leq \left\|~P_U^\perp~x\right\| + \left\|~y\right\| \\
        &\leq \lr{1+Cc_3\sqrt{T}}\left\|~P_U^\perp~x\right\| + Cc_1\sqrt{T}\lr{\left\|\varepsilon_{\Delta t}\lr{~x}\right\|+\left\|\varepsilon_A\lr{~P_U~x}\right\|} + C\sqrt{T}\left\|D_t(~x)\right\|\sqrt{\lambda_{\mathrm{max}}\lr{\hat{~J}^{-\intercal}\hat{~J}^{-1}-~I}}.
    \end{align*}
    The general case $~x_0\neq~0$ now follows similarly as in the proof of Theorem~\ref{thm:newrom}.

    To finish the argument, note that the velocity approximation error $\varepsilon_{\Delta t}\equiv 0$ is identically zero when the derivative operator $D_t = \dot{~\varphi_t}$ is exact.  It follows that the reprojected OpInf approximation error $\varepsilon_A\equiv 0$ is identically zero as well, since this enables exact recovery of the intrusive operator $\bar{~A}=\hat{~A}$.  The fact that the $L^2$-norm of the FOM velocity can be expressed as $\left\|\dot{~x}\right\| = \left\|~J\nabla H\lr{~x}\right\| = \left\|\nabla H\lr{~x}\right\|$ then shows that the OpInf ROM recovers exactly the same error bound as given in Theorem~\ref{thm:newrom}.  
\end{proof}

\begin{remark}
    While the error bound in Theorem~\ref{thm:newop} does not assume re-projection is employed, it has been expressed in a way which is useful in this case.  A slightly different variant of Theorem~\ref{thm:newop} with the same key terms can be established in the case that ``vanilla'' OpInf is assumed, c.f. \cite[Theorem 1]{Geng:2024} for a proof in the case of the least-squares Hamiltonian ROM \eqref{eq:symROM}.  However, without re-projection, there can be troublesome optimization error which pollutes the learned solution $\bar{~A}$, leading to a pre-asymptotic plateau in accuracy as more basis modes are added (c.f. \cite[Figure 10]{Geng:2024}).
\end{remark}


\section{Numerical Examples}\label{sec:numerics}



It is now appropriate to study the numerical performance\footnote{Code for reproducing the results in this Section can be found at \url{https://github.com/sandialabs/HamiltonianOpInf}.} of the variationally consistent ROM proposed in Sections~\ref{sec:method} and \ref{sec:opinf}.  The primary metric that will be used for this purpose is relative $\ell^2$ error between true and approximate state snapshots $~X,\tilde{~X}\in\mathbb{R}^{N\times n_t}$, which provides a normalized indicator of ROM performance:
\[ R\ell_2\lr{~X,\tilde{~X}} = \frac{\nn{~X-\tilde{~X}}_F}{\nn{~X}_F}. \]
In addition to error in the system state, signed error in the Hamiltonian function(al) will be measured along FOM and ROM trajectories, providing a complementary metric which detects violation of the governing physics.  Note that this deviation is significantly influenced by the choice of time integrator; the implicit midpoint method (trapezoid rule), which is a symplectic integrator and hence energy-conserving, is used for all ROMs in order to remove this effect.  Moreover, to facilitate visual comparison, this error will be measured in each case relative to the Hamiltonian value at the start of the ROM trajectory, which agrees with the FOM value of the Hamiltonian when centering is applied (c.f. Section~\ref{sbsec:int-centering}).  Since it will be instructive to compare the properties of POD bases used in building each ROM, the snapshot energy will also be evaluated: given a rank $r$ snapshot matrix $~X$ with singular values $\{\sigma_i\}_{i=1}^r$ and POD basis size $n\leq r$, the snapshot energy is
\[E_s\lr{~X,n} = \frac{\sum_{k=1}^n \sigma_k}{\sum_{k=1}^r \sigma_k}.\]
This quantity indicates how much of the total variance in the snapshots $~X$ is captured by the reduced basis $~U\in\mathbb{R}^{N\times n}$.


To isolate the effect of enforcing variational consistency in Hamiltonian ROMs, every example below will be linear in the state $~x$.  In fact, this simplified scenario is already enough to produce use cases which are considerably difficult for ROMs (c.f. Section~\ref{sbsec:linelastic}), and it will be shown that the modifications proposed here may be necessary for acceptable performance in these cases.  To provide an informative comparison, several ROMs are used in each of the given examples.  ``Intrusive G-ROM'' refers to the most obvious Galerkin ROM \eqref{eq:galROM}, and ``OpInf G-ROM'' refers to the nonintrusive version of this based on operator inference.  ``Intrusive H-ROM (inconsistent)'' refers to \eqref{eq:symROM}, while ``Intrusive H-ROM (consistent)'' refers to the proposed variationally consistent ROM \eqref{eq:newROM}.  Operator inference H-ROMs corresponding to the variationally consistent case are also computed as in Algorithm~\ref{alg:newOpInf} and denoted by ``OpInf H-ROM (original)'' and ``OpInf H-ROM (reprojected)'' to denote the absence respectively presence of re-projection (c.f. Section~\ref{sbsec:reproj}).  In most cases, the centering procedures from Sections~\ref{sbsec:int-centering} and \ref{sbsec:nonint-centering} are applied to conserve the correct energy level along ROM trajectories, and this is denoted with ``MC'' in plot legends.  
To further probe similarities and differences, each considered ROM is applied using three different POD basis constructions: ``Ordinary POD'' denotes the basis $~U$ formed using a simple SVD on snapshots $~X\approx~U~\Sigma~V^\intercal$; ``Cotangent Lift'' denotes the $~J$-equivariant algorithm mentioned in Section~\ref{sec:intro}, where a block diagonal basis of rank $2m$ is built as $~U=\mathrm{diag}\lr{~U_{qp},~U_{qp}}$ with $~U_{qp}$ defined by the rank $m$ SVD of a concatenated snapshot matrix $\lr{~Q|\alpha~P}$ formed from position snapshots $~Q$ and momentum snapshots $~P$ which are scaled by a constant factor $\alpha = \nn{~Q}_F/\nn{~P}_F$ to equalize their influence; ``Block $(q,p)$'' denotes the block diagonal mixed basis of rank $2m$ built from the separate rank $m$ SVDs of position and momentum snapshots, i.e. the Block $(q,p)$ basis is computed as $~U=\mathrm{diag}\lr{~U_{qq},~U_{pp}}$ where $~Q\approx~U_{qq}~\Sigma_{qq}~V_{qq}^\intercal$ and $~P\approx~U_{pp}~\Sigma_{pp}~V_{pp}^\intercal$ denote the relevant decompositions for snapshots $~Q,~P$ of position resp. momentum.  For completeness, the (scale-equalized) $~J$-equivariant ``Complex SVD'' algorithm described \cite{peng2016symplectic} is also presented in basis comparisons, though it is not considered in ROM building due to its near-identical performance to the cotangent lift.



\subsection{Toy Example: One-dimensional (1D) Wave}\label{sbsec:toy}
First, it is instructive to see a toy example where the issues with the variationally inconsistent ROM \eqref{eq:cotliftROM} do not arise.  Consider a 1D linear wave equation with constant speed 
$c\in\mathbb{R}$ and wavefunction $\phi:[0,s]\times [0,T]\to\mathbb{R}$,
\begin{equation}\label{eq:lin_wave_eqn}
\begin{split}
\phi_{tt}(s,t) &= c^2 \phi_{ss}(s,t), \qquad 0\leq s\leq l, \\
\phi(s,0) &= h(y(s)), \qquad \phi_t(s,0) = 0,
\end{split}
\end{equation}
where periodic boundary conditions are applied and where the initial condition $h:[0,\infty)\to\mathbb{R}$ is a cubic spline, 
\[ h(y) = \begin{cases} 1 - \frac{3}{2}y^2 + \frac{3}{4}y^3 & 0 \leq y \leq 1, \\ \frac{1}{4}\lr{2-y}^3 & 1<y\leq 2, \\ 0 & y > 2,
\end{cases} \qquad y(s) = \nn{s-\frac{1}{2}}. \]
This system is canonically Hamiltonian in the state variable $~x = \begin{pmatrix}q & p\end{pmatrix}^\intercal \eqqcolon \begin{pmatrix}\phi & \phi_t\end{pmatrix}^\intercal$ involving the wavefunction and its first time derivative, which can be verified by defining the Hamiltonian functional
\[ H\lr{~x} = \frac{1}{2}\int_0^l \lr{p^2 + c^2q_s^2}\, ds, \]
and writing \eqref{eq:lin_wave_eqn} in Hamiltonian form,
\begin{align}
    \dot{~x} = \begin{pmatrix}\dot{q} \\ \dot{p}\end{pmatrix} = \begin{pmatrix}0 & 1 \\ -1 & 0\end{pmatrix}\begin{pmatrix}-c^2\partial_{ss} & 0 \\ 0 & 1 \end{pmatrix}\begin{pmatrix}q \\ p\end{pmatrix} =~J\nabla H(~x).
\end{align}
Note that the gradient $\nabla H$ has been overloaded to mean the usual (Frechet) derivative of the functional $H$.  To apply Hamiltonian POD ROMs to this system, it is necessary to discretize it in a way which preserves its variational structure. One such way is to subdivide the spatial interval $[0,l]$ into $(M-1)$ subintervals of equal length, so that the position and momentum functions $q,p$ can be discretized along the $M$ grid points bounding these intervals.  Abusing notation, this produces a semidiscrete state $~x = \begin{pmatrix}~q & ~p\end{pmatrix}\in\mathbb{R}^N$ of dimension $N=2M$, with a corresponding discrete Hamiltonian
\[ H(~x) = \frac{1}{2}\sum_{i=1}^{N/2}\lr{ \lr{p^i}^2 + c^2 \frac{\lr{q^{i+1}-q^i}^2 + \lr{q^i-q^{i-1}}^2}{4\Delta x^2}}.\]
Applying the implicit midpoint rule for time integration then yields a fully discrete system for the state $~x_k$ at each time $k\Delta t$,
\begin{equation}\label{eq:disc_lin_wave_eqn}
\frac{~x_{k+1}-~x_{k}}{\Delta t} = ~J~A\lr{\frac{~x_{k+1}+~x_{k}}{2}} = \begin{pmatrix}~0 & ~I \\ -~I & ~0\end{pmatrix}\begin{pmatrix}-c^2~D_2 & ~0 \\ ~0 & ~I\end{pmatrix}\lr{\frac{~x_{k+1}+~x_k}{2}},
\end{equation}
where $~D_2$ denotes the usual circulant matrix resulting from a three-point stencil finite difference discretization of the 1D Laplace operator.  It is straightforward to check that the discrete Hamiltonian is exactly conserved along this evolution, so that \eqref{eq:disc_lin_wave_eqn} is the discrete Hamiltonian counterpart to \eqref{eq:lin_wave_eqn}.  For the examples here, the wave speed is fixed to $c=0.1$, the length to $l=1$, and the spatial domain is divided into $M=500$ equally sized intervals to yield a discrete state vector $~x$ of dimension $N=1000$.  

\subsubsection{Time-reproductive simulation}
To probe the performance of the various ROMs in a reproductive setting, snapshot data are collected every $\Delta t=0.02$ seconds on the interval $[0,10]$ and placed into a matrix $~X\in\mathbb{R}^{1000\times 501}$.  Applying different POD basis constructions in centered and uncentered cases leads to the snapshot energies and projection errors in Figure~\ref{fig:1D-POD-Energy}, where it is clear that adding basis modes to each method produces the slow-but-smooth error decay characteristic of problems with a large Kolmogorov $n$-width \cite{Pinkus:1985}.

\begin{figure}[htb]
    \centering
    \includegraphics[width=\textwidth]{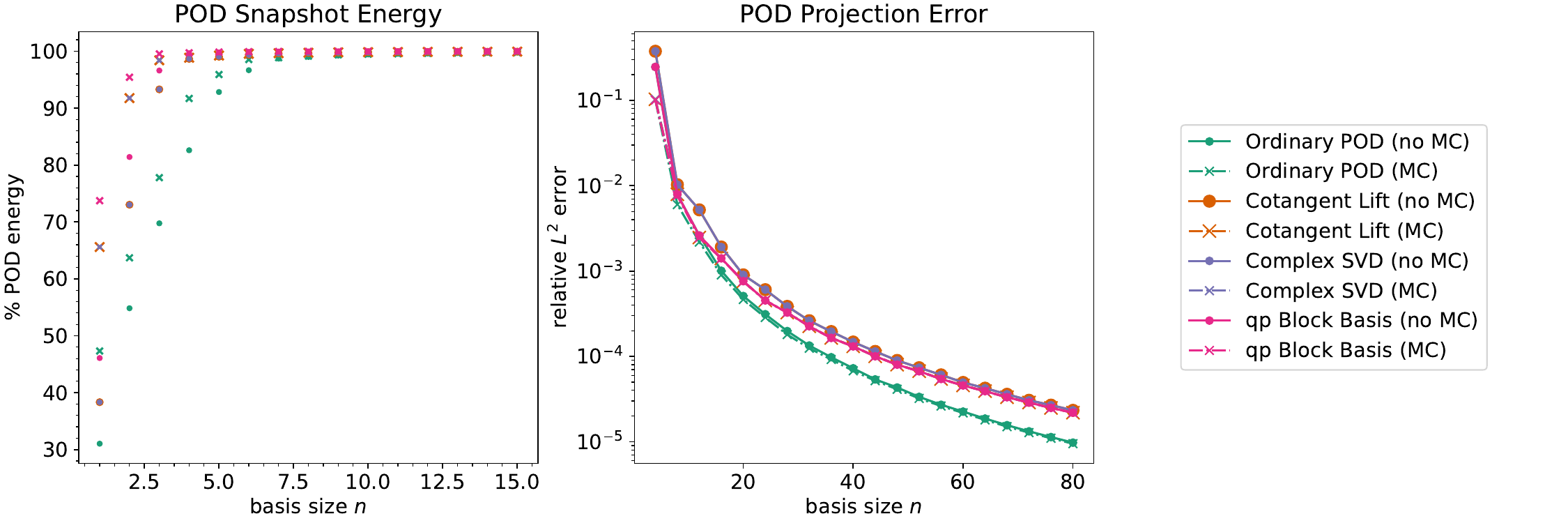}
    \caption{POD basis snapshot energies and projection errors corresponding to the toy 1D wave example in Section~\ref{sbsec:toy}.}
    \label{fig:1D-POD-Energy}
\end{figure}

Results of applying each ROM to the POD-projected version of \eqref{eq:disc_lin_wave_eqn} under the centered approximation $\tilde{~x}=~x_0+~U\hat{~x}$ are displayed in Figure~\ref{fig:1D-T10-Errors}, along with the errors in the Hamiltonian along each evolution.  Clearly, any algorithm is sufficient for producing an accurate ROM in this case, although there is a modest benefit to applying a Hamiltonian ROM. Interestingly, the variational inconsistency introduced by \eqref{eq:symROM} does not seem to harm ROM performance here, likely because this example is so simple.  Notably, the Block $(q,p)$ basis appears to lead to the most accurate ROMs, even though it did not have the lowest projection errors in Figure~\ref{fig:1D-POD-Energy}. On the OpInf side, it is clear that the re-projected OpInf ROMs match their intrusive counterparts exactly, as expected, while the vanilla OpInf ROMs are either very close to their intrusive counterparts or slightly unstable in the case of the OpInf G-ROM with block $(q,p)$ basis.  It is remarkable that both vanilla OpInf ROMs are slightly superconvergent in the Cotangent Lift case, so that ``non-Markovian'' effects in the learned operators are actually helpful here.  

The Hamiltonian errors in Figure~\ref{fig:1D-T10-Errors} appear as expected.  While all ROMs begin at the correct energy level due to the centering procedure employed, only the Hamiltonian ROMs are able to maintain this energy level to order $10^{-13}$ throughout their evolution.  Again, it does not seem that variational inconsistency harms performance here, as the inconsistent H-ROM preserves the Hamiltonian to the same degree as the consistent ones.  This is not exceptionally surprising: all that is required for energy conservation is skew-symmetry of the reduced symplectic matrix $\hat{~J}$, which is obviously satisfied by the matrix $\hat{~J}=~U^\intercal~J~U$ used in the inconsistent case.  Combined with the error plots in Figure~\ref{fig:1D-T10-Errors}, it is clear that there is remarkably little difference in overall performance between inconsistent and consistent H-ROMs in this example, confirming that it is not sophisticated enough to draw out such differences.

\begin{figure}[htb]
    \centering
    \begin{minipage}{\textwidth}
        \includegraphics[width=\textwidth]{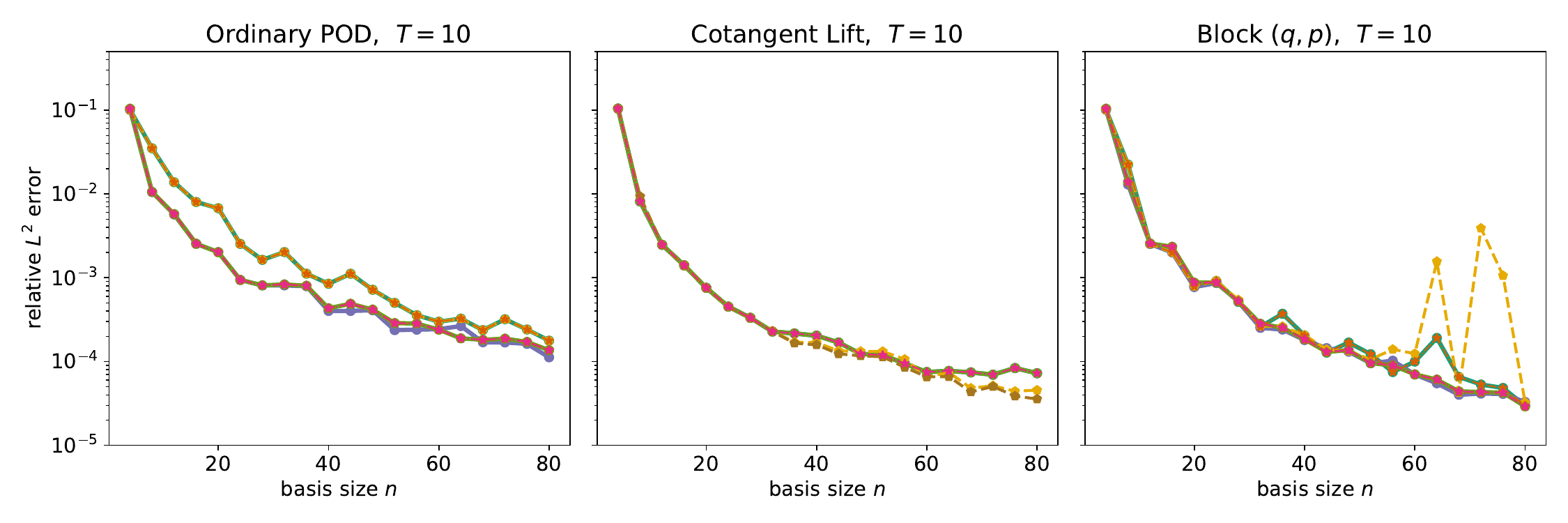}    
    \end{minipage}%
    \\
    \begin{minipage}{0.95\textwidth}
        \includegraphics[width=\textwidth]{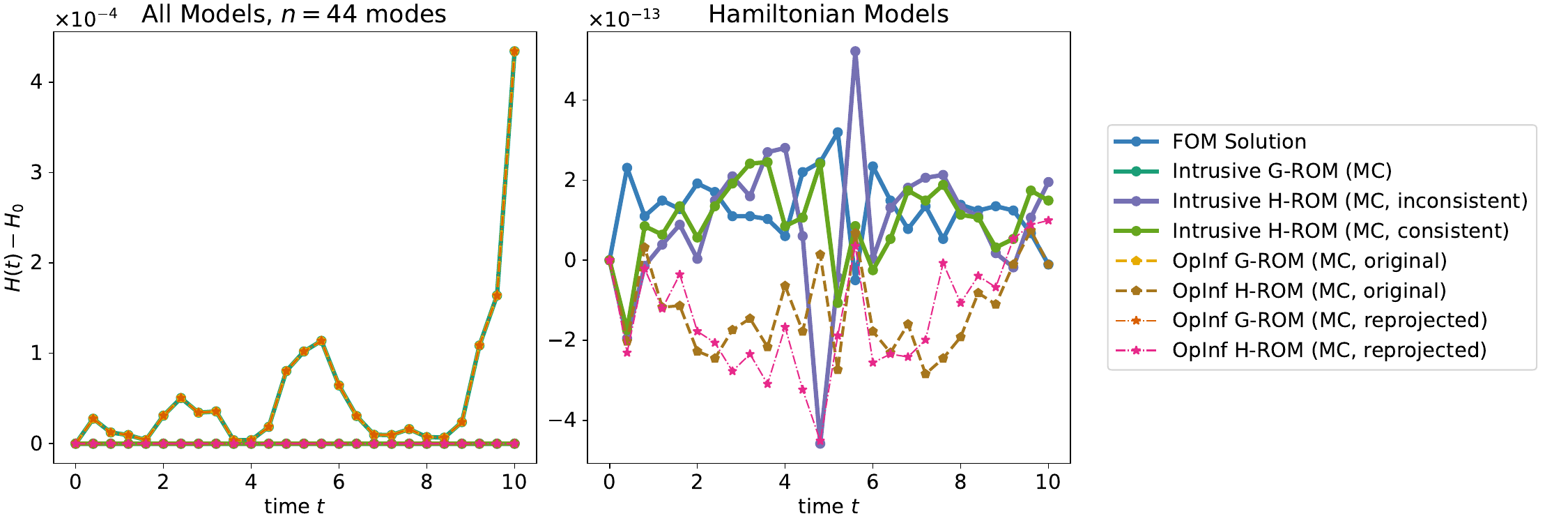}
    \end{minipage}
    \caption{ROM errors (top) and Hamiltonian errors (bottom) corresponding to the time-reproductive toy 1D wave example in Section~\ref{sbsec:toy}.  Hamiltonian error plots display ROMs using Ordinary POD basis.}
    \label{fig:1D-T10-Errors}
\end{figure}


\subsubsection{Time-predictive simulation}
While each ROM performed similarly in the time-reproductive case, differences start to emerge when the domain of application is extended to $[0,100]$ instead of the training domain $[0,10]$.  In this case, the same ROMs from above are tested on $[0,100]$ in intervals of $\Delta t=0.1$, a five-fold increase in time step and ten-fold increase in integration domain.  It is easy to see from the ROM errors and Hamiltonian errors in Figure~\ref{fig:1D-T100-Errors} that both the choice of POD basis and choice of ROM algorithm are important to performance.  Remarkably, the choice of centered approximation is also highly influential: experiments in \cite[Figure 6]{gruber2023canonical} show that OpInf G-ROMs will become unstable on this problem if centering is not applied.  

It is interesting to observe that there is more separation between the errors of the G-ROMs and H-ROMs in the time-predictive case, which makes sense as only the H-ROMs enforce the Hamiltonian structure necessary for correct dynamical evolution.  Remarkably, although the Block $(q,p)$ basis is again the most performant, it apparently leads to unstable G-ROMs as more modes are added.  Similar to before, it seems that re-projection makes little difference on the OpInf ROMs using the Ordinary POD basis, although its effect is noticeable in the other bases.  It is further interesting to note that the inconsistent H-ROM exhibits very slightly worse ROM error decay in the Ordinary POD case, perhaps displaying a consequence of its variational inconsistency.  The Hamiltonian errors displayed in Figure~\ref{fig:1D-T100-Errors} are remarkably similar to what is seen in the reproductive case, with all Hamiltonian ROMs conserving energy to the same order as the FOM throughout the entire length of time integration.

\begin{figure}[htb]
    \centering
    \begin{minipage}{\textwidth}
        \includegraphics[width=\textwidth]{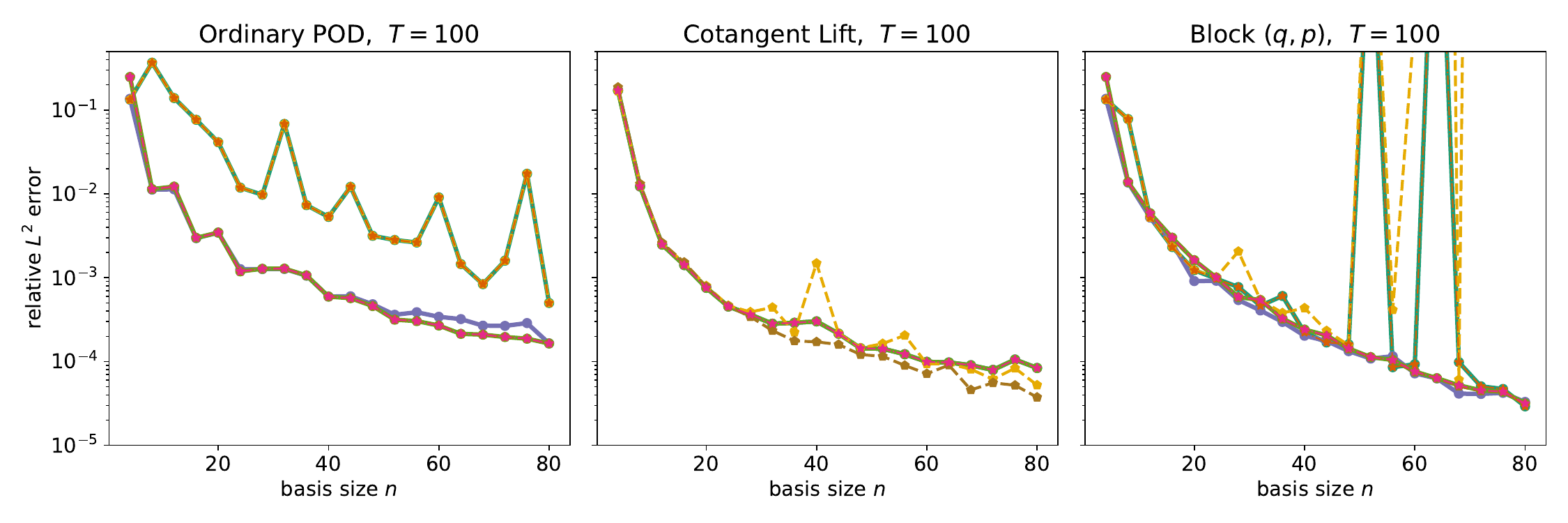}    
    \end{minipage}%
    \\
    \begin{minipage}{0.95\textwidth}
        \includegraphics[width=\textwidth]{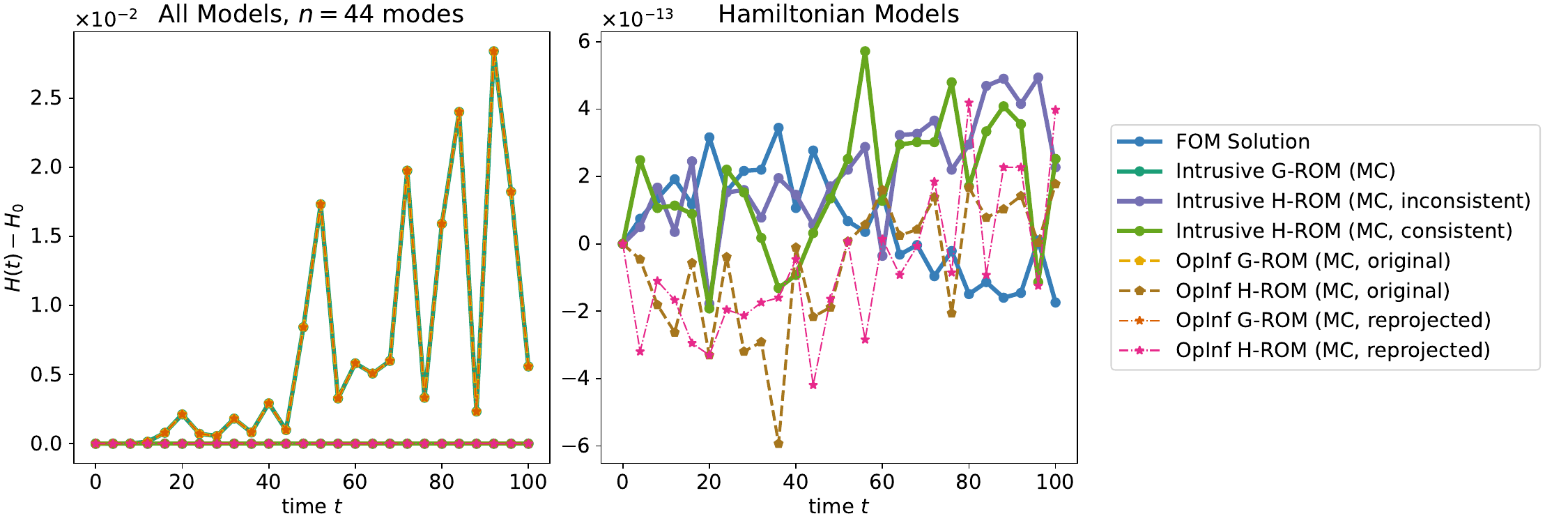}
    \end{minipage}
    \caption{ROM errors (top) and Hamiltonian errors (bottom) corresponding to the time-predictive toy 1D wave example in Section~\ref{sbsec:toy}.  Hamiltonian error plots display ROMs using Ordinary POD basis.}
    \label{fig:1D-T100-Errors}
\end{figure}

\subsection{Three-dimensional (3D) Linear Elastic Examples}\label{sbsec:linelastic}
The remaining examples are moderately sized problems from 3D linear elasticity, which has the following equations of motion in terms of a displacement vector $~q\in\mathbb{R}^3$, material density $\rho>0$, and Cauchy stress tensor $~\sigma\in\mathbb{R}^{3\times 3}$:
\begin{equation} \label{eq:3delasticity}
    \rho \ddot{~q} = \nabla \cdot  ~\sigma, \hspace{0.5cm} \text{on } \Omega \in \mathbb{R}^3.
\end{equation}
It will be assumed that the material in question is elastic and follows Hooke's law, so that the stress tensor $~\sigma$ has the explicit expression
\begin{equation*}
    ~\sigma = \lambda \text{Tr}(~\epsilon)~I + 2\mu ~\epsilon,
\end{equation*}
in terms of Lam\'{e} coefficients $\lambda, \mu >0$ and the symmetric strain tensor $~\epsilon \coloneqq \frac{1}{2} \left[ \nabla \bb{q} + (\nabla \bb{q})^\intercal \right]$.  It can be shown (see \cite{Marsden_BOOK_1998}) that the Hamiltonian functional for \eqref{eq:3delasticity} can be expressed concisely in terms of (noncanonical) position and velocity variables $\lr{~q,\dot{~q}}$: 
\begin{equation} \label{eq:3d_ham_lin_elastic}
    H(~q, \dot{~q}) = \frac{1}{2}\int_{\Omega} \lr{\rho |\dot{~q}|^2  + 
    \lambda [\text{Tr}(~\epsilon)]^2 + 2\mu \nn{~\epsilon}^2} dV.
\end{equation}
It is interesting to note that, in 1D, equation \eqref{eq:3d_ham_lin_elastic} reduces to the linear wave equation \eqref{eq:lin_wave_eqn} considered in Section \ref{sbsec:toy}. 
Considering 3D problems introduces challenging solution features that are difficult for ROMs to capture, e.g., waves propagating in multiple spatial directions over potentially complex 3D geometries (e.g., Figures \ref{fig:sm_meshes}(a) and (b)), which can complicate wave interaction.  For physical realism, these model geometries are assumed to be made of steel, and the material parameters $E = 200$ GPa (Young's modulus), $\nu = 0.25$ (Poisson's ratio), and $\rho = 7800$ kg/m$^3$ (density) lead to straightforward calculations for the Lam\'{e} coefficients in the stress tensor $~\sigma$ via the formulas $\lambda = \frac{E\nu}{(1+\nu)(1-2\nu)}$ and $\mu = \frac{E}{2(1+\nu)}$.  It is well-known \cite{Dhari:2024} that the Courant-Friedrichs-Lewy (CFL) condition for \eqref{eq:3delasticity} is given by 
\begin{equation} \label{eq:dt_wave}
\Delta t_w \leq \frac{h
}{c},
\end{equation}
where $h$ is the approximate average mesh resolution, and
$c = \sqrt{\frac{LM}{\rho}}$ is the speed of sound.
Here, $LM$ denotes the longitudinal elastic modulus, ``or P-wave modulus” as it is commonly referred to, which can be calculated from the elastic coefficients using the formula
\begin{equation} \label{eq:LM}
LM := \frac{E(1-\nu)}{(1+\nu)(1-2\nu)}.  
\end{equation}
It is well-known that time-steps less than the one given by \eqref{eq:dt_wave} are sufficient to resolve the highest frequency waves that the medium can propagate.  In view of this, the estimate \eqref{eq:dt_wave} is used as a guide to inform the FOM time-steps used in the numerical experiments presented below.



As per common practice in the field of solid mechanics, equation \eqref{eq:3delasticity} is discretized in space using the finite element method (FEM), yielding a semidiscrete system of the form \eqref{eq:Lagrangian} with $~f = ~0$,
where (again overloading notation) $~q \in \mathbb{R}^M$ is the discretized displacement field while $~M \in \mathbb{R}^{M \times M}$ and $~K \in \mathbb{R}^{M \times M}$ are the mass and stiffness matrices of the scheme, respectively.  Letting $~p \coloneqq ~M \dot{~q}$ define the momentum leads to the discrete Hamiltonian state $~x := \begin{pmatrix}~q & ~p \end{pmatrix}^\intercal \in \mathbb{R}^{N}$, and allows \eqref{eq:Lagrangian} to be written as the following canonical Hamiltonian system: 
\begin{equation} \label{eq:sm_first_order_sys}
    \dot{~x} = \begin{pmatrix} \dot{~q} \\ \dot{~p} \end{pmatrix} 
    = \left( \begin{array}{cc}
    ~0 & ~I \\
    -~I & ~0 
    \end{array}\right) \left( \begin{array}{cc}
    ~K & ~0 \\
    ~0 & ~M^{-1} 
    \end{array}\right)\begin{pmatrix} ~q \\ ~p \end{pmatrix}
    = ~J~A~x,
\end{equation}
where $~A=\nabla H$ is the discrete Hamiltonian gradient and $H$ is the quadratic discrete Hamiltonian,
\begin{equation*}
    H(~x) = \frac{1}{2} \left(~q^\intercal ~K~q + ~p^\intercal ~M^{-1}~p \right).
\end{equation*}



\begin{figure}
\begin{subfigure}[t]{.3\textwidth}\centering
  \includegraphics[width=\textwidth]{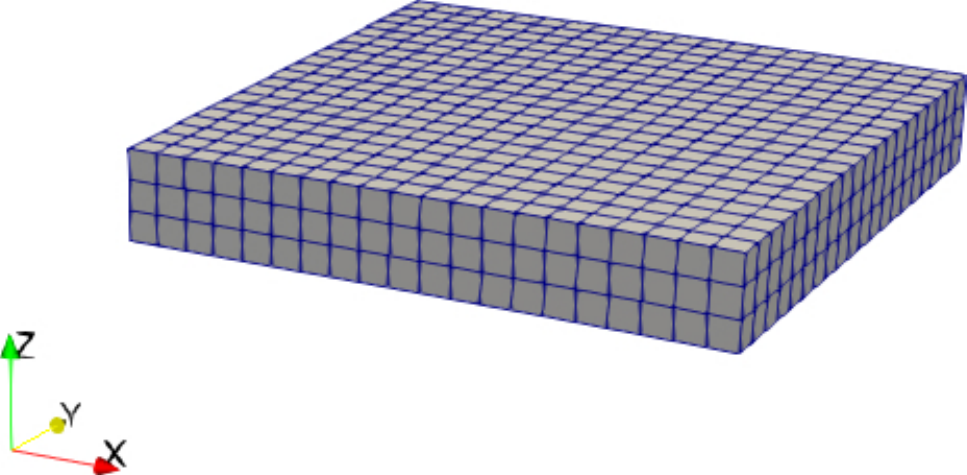}
  \caption{Cantilever Plate}
\end{subfigure}%
\hfill
\begin{subfigure}[t]{.3\textwidth}\centering
  \includegraphics[width=\textwidth]{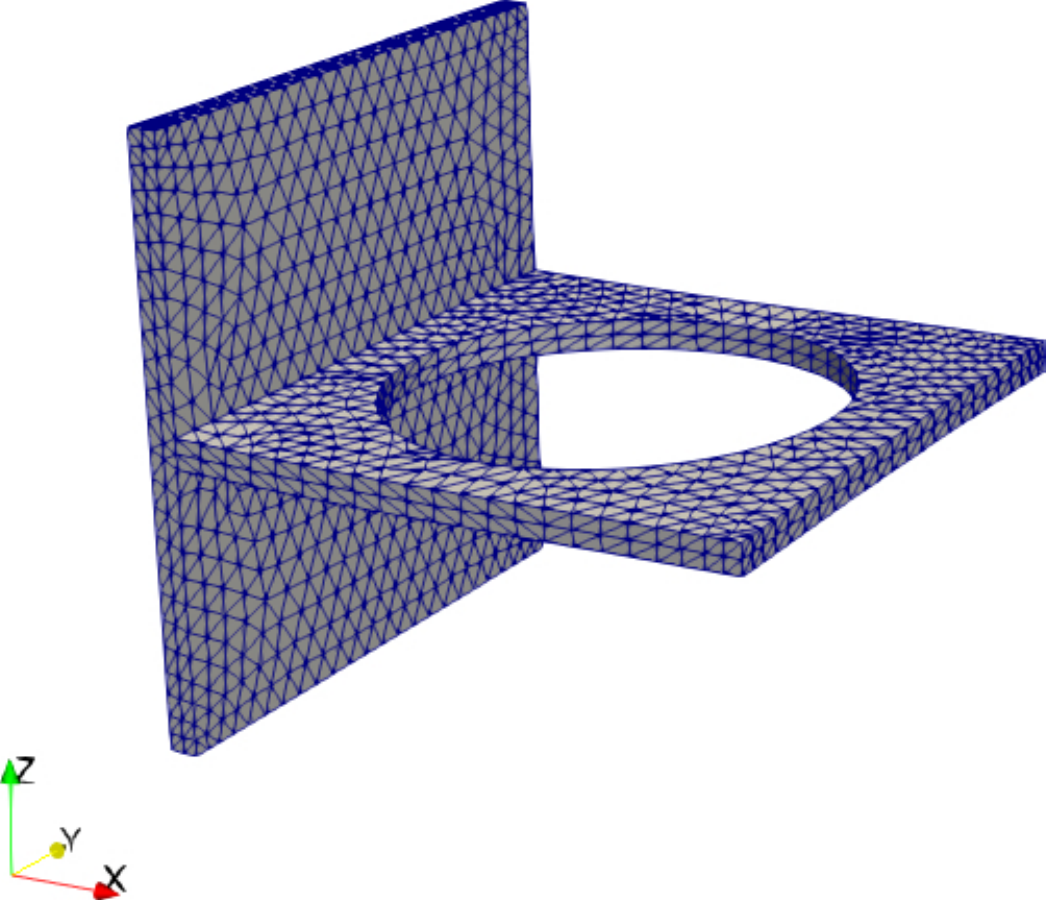}
  \caption{Flexible Bracket}
\end{subfigure}%
\hfill
\begin{subfigure}[t]{.3\textwidth}\centering
    \begin{tikzpicture}
    \point{begin}{0}{0};
    \point{middle}{2.5}{0};
    \point{end}{5}{0};
   \point{end2}{5}{-0.75}
    \beam{2}{begin}{end};
    \support{3}{begin}[-90];
    \load{1}{end}[-90];
    \notation{1}{end2}{$\dot{q}_{3}(0)=100$ m/s}[left];
    \foreach [evaluate={\in=180+\b*2}] \b in {5,10,...,30}{
      \draw[-, ultra thick] (begin) to[out=0,in=\in] (\b:5+\b*0.01);
    }
    \end{tikzpicture}
  \caption{One-dimensional cartoon}
\end{subfigure}
\caption{Geometries and meshes for the 3D linear elastic examples, along with a one-dimensional cartoon illustrating the problem setup.} 
\label{fig:sm_meshes}
\end{figure}
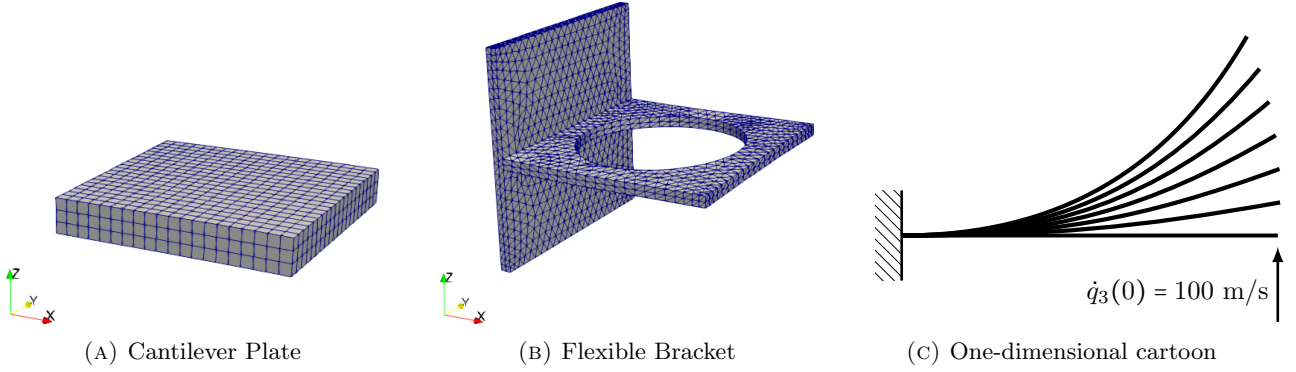

\subsubsection{Cantilever Plate}\label{sbsec:plate}
The first test case considered presently is a classical solid mechanics benchmark, also considered in \cite{gruber2023canonical}, involving a vibrating rectangular cantilever plate of size 0.2 $\times$ 0.2 $\times$ 0.03 meters, i.e. with spatial domain $\Omega = (0, 0.2) \times (0,0.2) \times (0, 0.03) \in \mathbb{R}^3$ (see Figure \ref{fig:sm_meshes}(a)).  Let $\bb{s}=\begin{pmatrix}s^1 & s^2 &s^3\end{pmatrix}^\intercal \in \mathbb{R}^3$ denote the position vector.  The left side of the plate is clamped in this benchmark, and a homogeneous Dirichlet boundary condition $~q = ~0$ is imposed on the left boundary $\Gamma_l:=\{s^2, s^3 \in \bar{\Omega}: s^1 = 0\}$.  The remaining boundaries of the domain $\Omega$ satisfy homogeneous Neumann boundary conditions, indicating that these boundaries are free surfaces.  An initial velocity of 100 m/s in the $s^3$-direction is prescribed on the right boundary of the domain, $\Gamma_r:=\{s^2,s^3 \in \bar{\Omega}:s^1 = 0.2\}$, so that $\dot{~q}\lr{~s,0} = (0, 0, 100)^\intercal$ for $s\in\Gamma_r$.  Figure~\ref{fig:sm_meshes}(c) displays a one-dimensional cartoon illustrating the problem setup, whereby the initial velocity perturbation will cause the plate to vibrate and undergo a flapping motion.  This vibration causes waves to form inside the plate, propagating in complex patterns across all three coordinate directions.  In particular, this represents a challenging test case for Hamiltonian model reduction which previous methods are unable to reliably solve (c.f. \cite[Figure 29]{gruber2023canonical}).


Building the FOM corresponding to system \eqref{eq:sm_first_order_sys} is accomplished using the open-source\footnote{Albany-LCM is available on Github at the following URL: \url{https://github.com/sandialabs/LCM}.} Albany-LCM multi-physics code base \cite{albany, Mota:2017, Mota:2022}.  The domain $\Omega$ is discretized with a uniform mesh of $20 \times 20 \times 3$ hexahedral elements, as shown in Figure \ref{fig:sm_meshes}(a), and the FOM system \eqref{eq:Lagrangian} is advanced forward from time $t=0$ to time $t = 2 \times 10^{-3}$ s to generate snapshots; here, a symplectic implicit Newmark time-stepping scheme with parameters $\beta = 0.25$, $\gamma = 0.5$, and time-step $\Delta t = 1.0 \times 10^{-7}$ s is used, and snapshots are collected every $1.0 \times 10^{-5}$ s.  It was verified that 
the time-step used in the FOM integration is below the time-step required to resolve the wave ($1.8\times 10^{-6}$ s, from \eqref{eq:dt_wave}).
The transformation $~p=~M\dot{~q}$ is applied to generate momentum data, and the resulting 201 snapshots of the discrete state $~x$, each of length 10584, are used to build the various POD bases.  Plots of the relevant snapshot energies and projection errors are shown in Figure~\ref{fig:Plate-POD-Energy}.  Notice in particular that the $~J$-equivariant constructions, i.e. the Cotangent Lift and Complex SVD, produce basis which are noticeably less expressive than the Ordinary POD basis, leading to projection errors which are several orders of magnitude higher after about 80 POD modes.

\vspace{1pc}

\noindent\textbf{Time-reproductive simulation.}
Again, the first test presented here involves a time-reproductive case, where the goal of each ROM is to reproduce the snapshots used for training.  The results of this experiment are shown in Figure~\ref{fig:Plate-T0002-Errors}, where it is immediately clear that preserving a variationally consistent Hamiltonian structure is necessary for adequate performance.  In each case it is observed that the inconsistent H-ROM produces order-one errors regardless of basis size, confirming that the least-squares projection used to enforce Hamiltonian structure is detrimental to ROM accuracy.  Conversely, the variationally consistent H-ROM in both its intrusive and nonintrusive versions converges smoothly at the rates expected based on the projection errors in Figure~\ref{fig:Plate-POD-Energy}.  Importantly, it can be seen that all ROMs built using the Cotangent Lift basis (which inherit variational consistency automatically) are unusable in this case, as this basis is not expressive enough to capture the variance in the solution snapshots.  

Besides its implications on ROM accuracy, notice that preservation of Hamiltonian structure is also crucial for ROM stability: each pictured G-ROM is unstable with the addition of basis modes even though some choices of reduced dimension $n$ lead to relatively accurate solutions.  This is an illustration of the commonly observed accuracy vs. stability trade-off seen in structure-preserving methods.  While there is nothing which prevents a straightforward G-ROM from being as accurate as any H-ROM (although this is not observed here), Lyapunov stability is guaranteed only by the H-ROMs.  The same goes for energy conservation:  Figure~\ref{fig:Plate-T0002-Errors} shows clearly that each Hamiltonian ROMs is conservative while the G-ROMs are not, and also that the variationally inconsistent H-ROM \eqref{eq:symROM} is almost as conservative as the variationally consistent one despite its very poor accuracy.


\begin{figure}[htb]
    \centering
    \includegraphics[width=\textwidth]{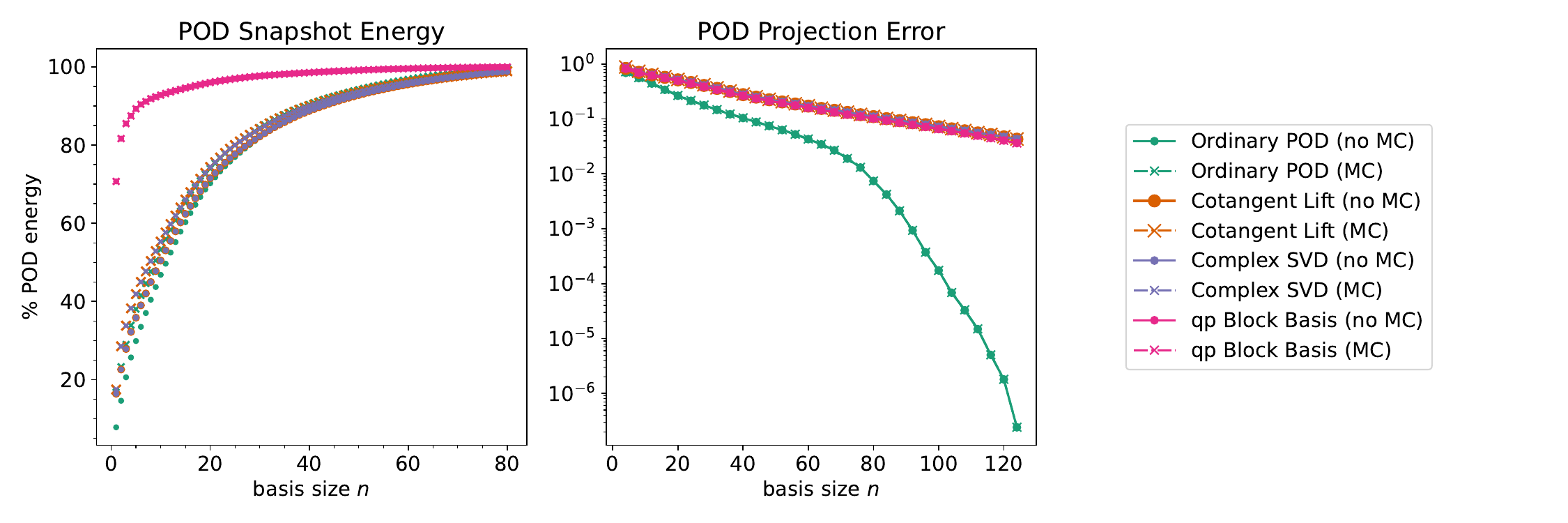}
    \caption{POD basis snapshot energies and projection errors corresponding to the cantilever plate example in Section~\ref{sbsec:plate}.}
    \label{fig:Plate-POD-Energy}
\end{figure}

\begin{figure}[htb]
    \centering
    \begin{minipage}{\textwidth}
        \includegraphics[width=\textwidth]{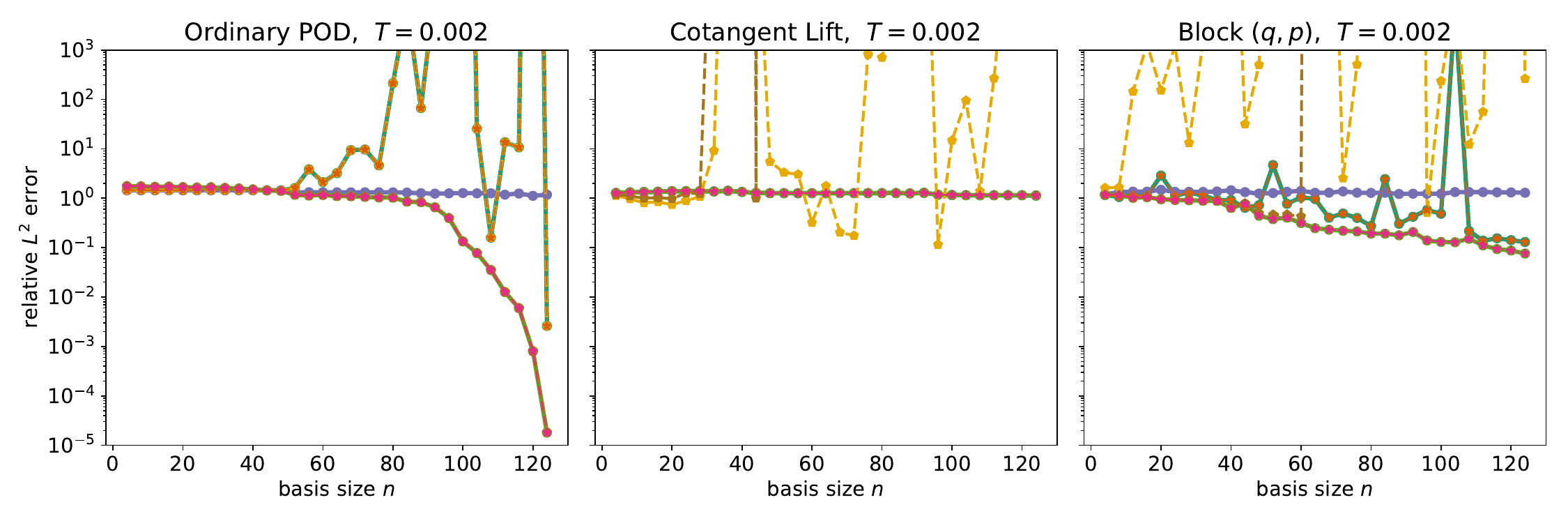}    
    \end{minipage}%
    \\
    \begin{minipage}{\textwidth}
        \includegraphics[width=\textwidth]{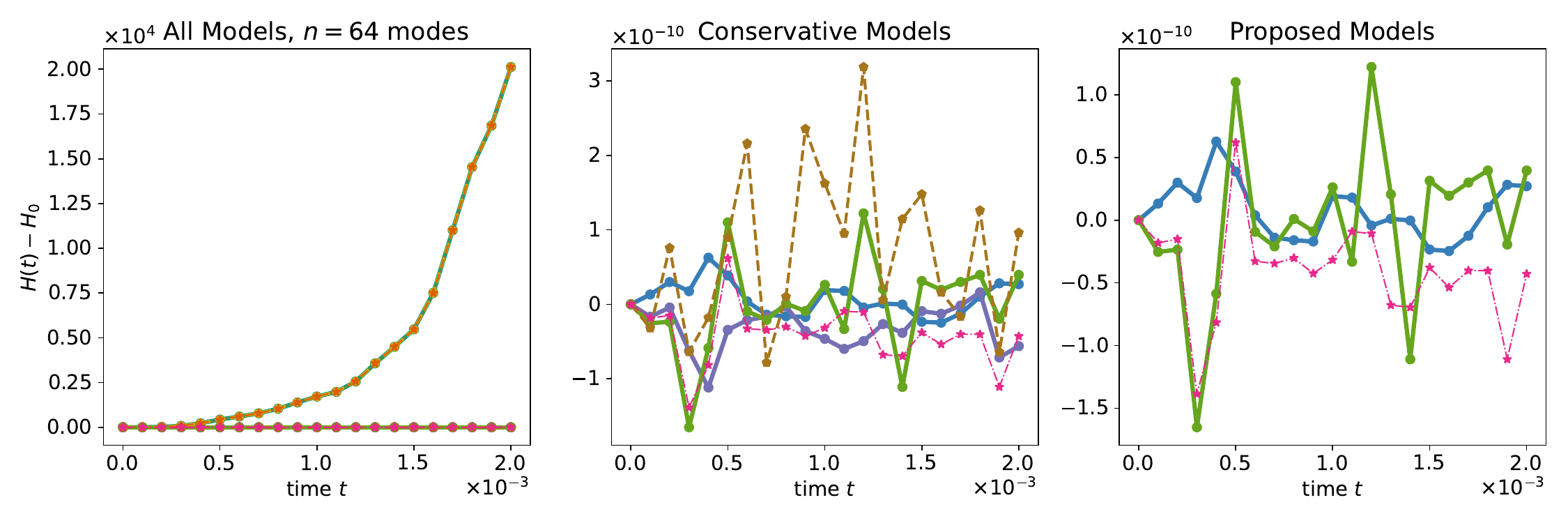}
    \end{minipage}%
    \\
    \begin{minipage}{\textwidth}
        \includegraphics[width=\textwidth]{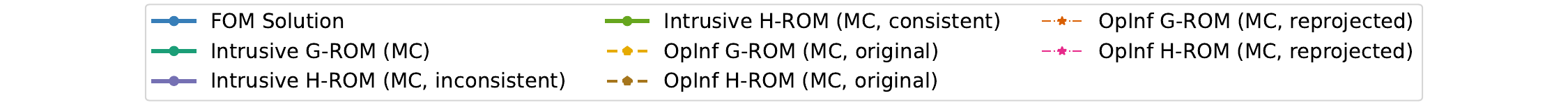}
    \end{minipage}
    \caption{ROM state errors (top) and Hamiltonian errors (bottom) corresponding to the time-reproductive case of the cantilever plate example in Section~\ref{sbsec:plate}.  Hamiltonian error plots display ROMs using Ordinary POD basis.}
    \label{fig:Plate-T0002-Errors}
\end{figure}

\vspace{1pc}

\noindent\textbf{Time-predictive simulation.}
Beyond the reproductive case seen above, it is worthwhile to examine what happens when the ROMs describing the cantilever plate are simulated well beyond where they have been trained.  Figure~\ref{fig:Plate-T0005-Errors} shows the results of this in the case that the final time is $T=5\times 10^{-3}$ with the same time step $\Delta t = 1.0 \times 10^{-5}$, amounting to a test interval which is $2.5\times$ longer than the training interval.  While this problem is periodic, note that these ranges are well below the period length of 0.02. This means that the dynamics on the unseen region $[0.002,0.005]$ are substantially different than those seen by the ROMs during training, making this problem more difficult than the similar predictive problem in \cite[Section 5.5]{gruber2023canonical}.  The error plots in Figure~\ref{fig:Plate-T0005-Errors} confirm that only the Hamiltonian models are capable of stable results on this test case, and that only the variationally consistent H-ROMs exhibit some semblance of convergence to a realistic solution.  This is further confirmed by examining the trajectory of a single point displayed in Figure~\ref{fig:Plate-T0005-Trajectories}, which visibly highlights the artifacts in the inconsistent H-ROM solution.   Again, it is clear from Figure~\ref{fig:Plate-T0005-Errors} that the Cotangent Lift basis is unusable here, although it is remarkable that each of the non-re-projected (i.e. original) OpInf ROMs is highly unstable despite the $~J$-equivariance of the reduced basis.  The plots of Hamiltonian error tell a similar story: while each G-ROM artificially adds energy to the system over time, the H-ROMs are conservative to the same order as the FOM.  Moreover, it is clear that variational consistency does not have much impact on conservation despite its large impact on accuracy.


\begin{figure}[htb]
    \centering
    \includegraphics[width=\textwidth]{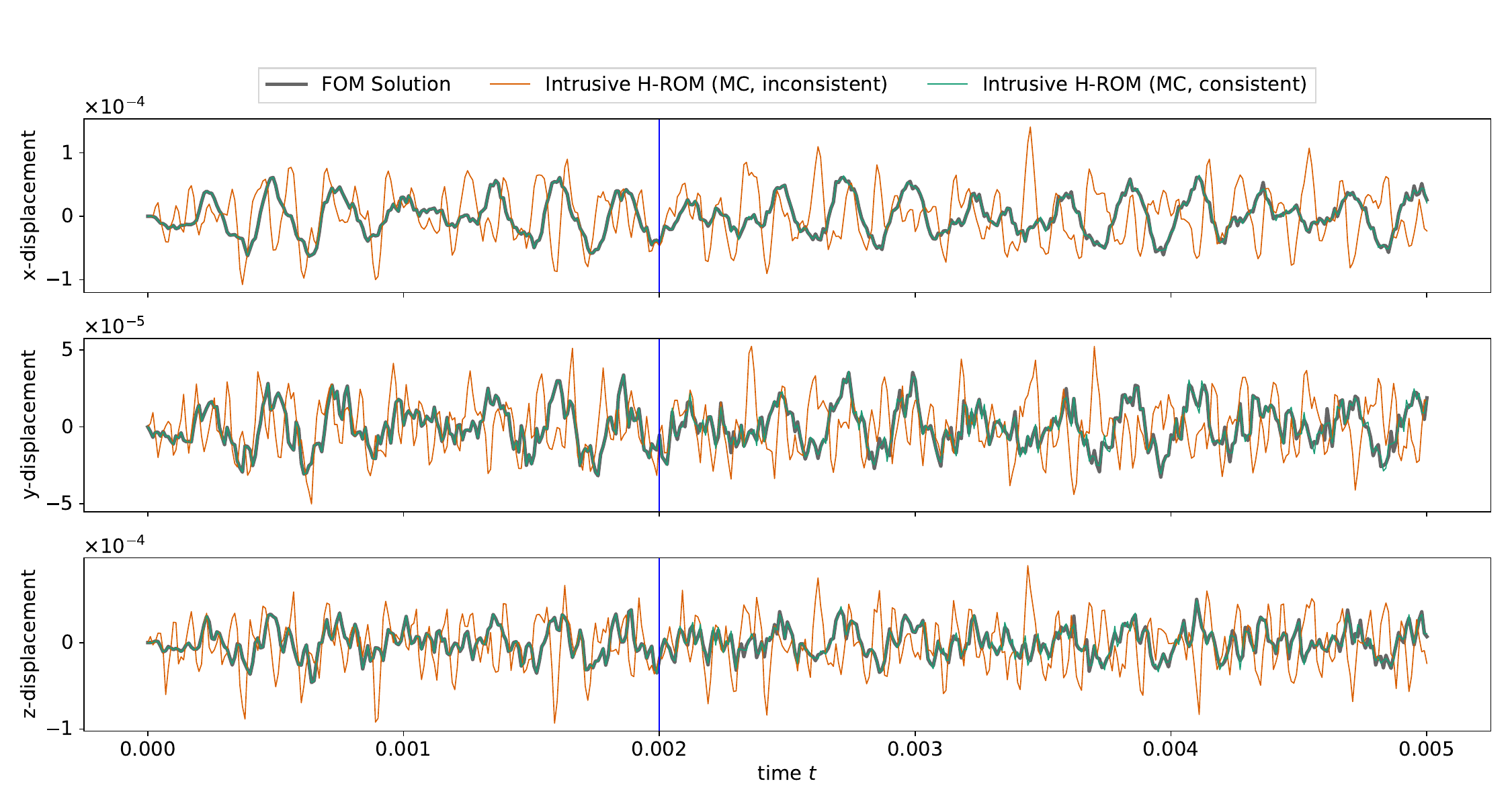}
    \caption{Time-histories of the $x$--, $y$-- and $z$--displacement solutions to the Cantilever Plate example at a single node with global node ID 1000 and spatial coordinates $(x,y,z) = (0.12, 0.05, 0.02)$.  The vertical blue line denotes the end of the training interval.  All ROMs were computed with an Ordinary POD basis of $n=100$ modes.  Note that the trajectories of the variationally consistent OpInf ROMs (not pictured) lie directly on top of the variationally consistent intrusive H-ROM.}
    \label{fig:Plate-T0005-Trajectories}
\end{figure}



\begin{figure}[htb]
    \centering
    \begin{minipage}{\textwidth}
        \includegraphics[width=\textwidth]{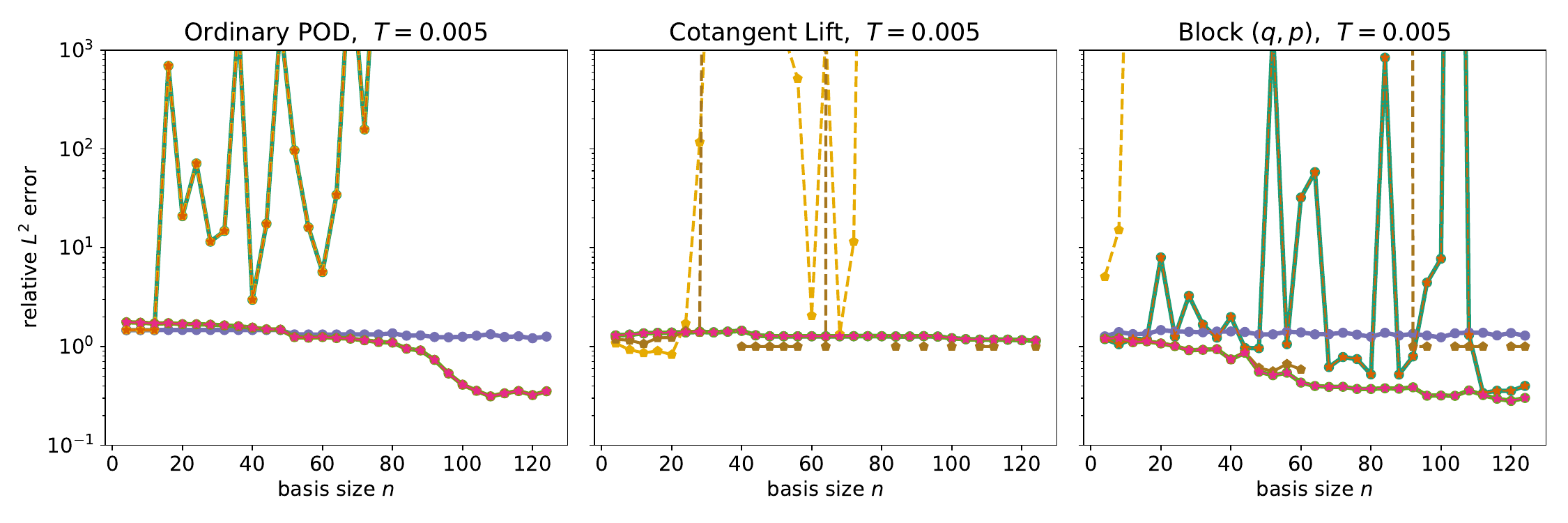}    
    \end{minipage}%
    \\
    \begin{minipage}{\textwidth}
        \includegraphics[width=\textwidth]{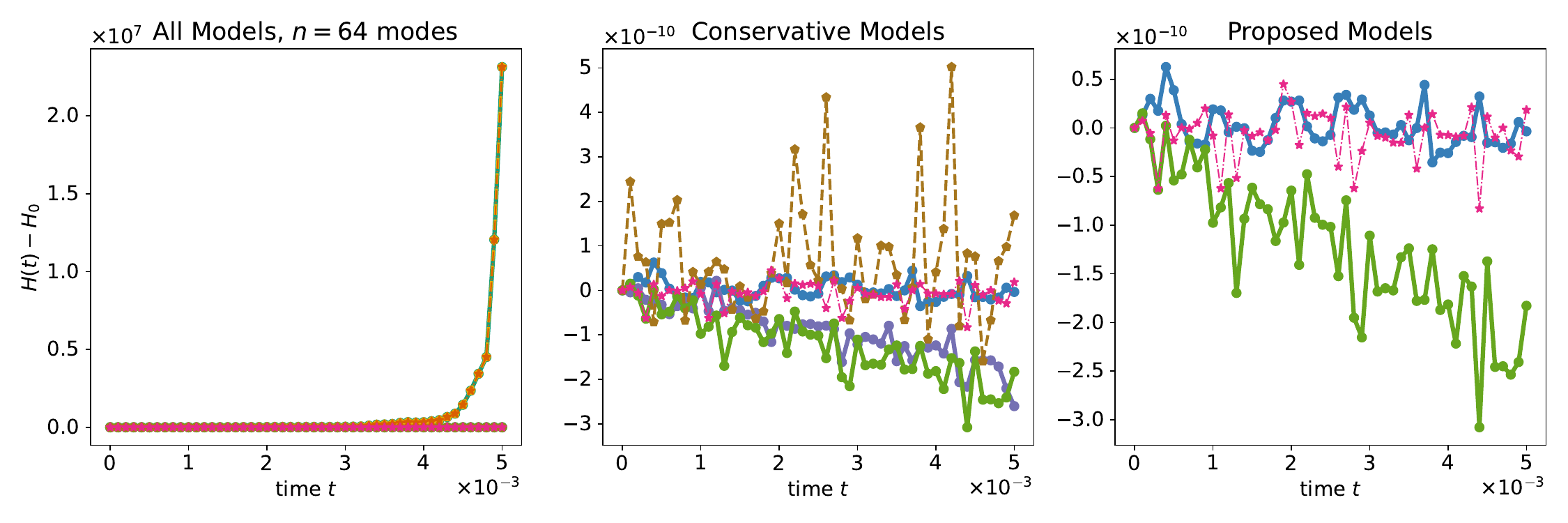}
    \end{minipage}%
    \\
    \begin{minipage}{\textwidth}
        \includegraphics[width=\textwidth]{figs/PlateHamilLegend.pdf}
    \end{minipage}
    \caption{ROM state errors (top) and Hamiltonian errors (bottom) corresponding to the time-predictive case of the cantilever plate example in Section~\ref{sbsec:plate}.  Hamiltonian error plots display ROMs using Ordinary POD basis.}
    \label{fig:Plate-T0005-Errors}
\end{figure}

\subsubsection{Flexible Bracket}\label{sbsec:bracket}

The final example presented here involves a flexible bracket clamped at one end.  Although the experimental setup 
is similar to the cantilever plate example, the problem is posed on a more complex geometry (see Figure \ref{fig:sm_meshes}(b)) featuring a circular hole whose presence weakens the bracket, as well as complicates the propagation and interaction of 3D waves propagating throughout the bracket.  Additionally, both the final integration time as well as the time interval at which snapshots are collected are widened considerably.  This has the effect of implicitly damping the oscillations in the metal, effectively making the dynamics easier to capture with a POD basis.
It is interesting to see
how this change, as well as the natural change caused by the inclusion of a more challenging mesh topology, affects ROM performance.  To facilitate further comparison, the straightforward Lagrangian ROM (c.f. \cite{Carlberg:2015}) arising from Galerkin projection onto a standard POD basis for the displacement $~q$ will be presented here alongside the Hamiltonian methods previously considered.  Note that this method returns approximations to the displacement $~q$, velocity $\dot{~q}$, and acceleration $\ddot{~q}$, and therefore the approximate momentum $\tilde{~p}=~M\dot{\tilde{~q}}$ is post-processed from the velocity approximation.

The domain on which the problem is posed is 20.5 cm wide in the $x$--direction and 20 cm wide in the $y$-- and $z$--directions.  The geometry is discretized using an unstructured tetrahedral mesh of 9640 elements, each with a mesh increment of approximately 5.8 mm  (Figure \ref{fig:sm_meshes}(b)).
 The FOM system \eqref{eq:Lagrangian} is advanced forward from time $t=0$ to time $t = 2 \times 10^{-2}$ s using $\Delta t = 1.0 \times 10^{-7}$ s,
and the larger sampling time $\Delta t = 1.0 \times 10^{-4}$ s is used for snapshot collection. It was verified that the time-step used to simulate the FOM was below the minimum time-step required to resolve the wave ($1.05\times 10^{-6}$s, from \eqref{eq:dt_wave}). 
As before, the transformation $~p=~M\dot{~q}$ is applied to generate momentum data, and the resulting 201 snapshots of the discrete state $~x$, each of length 15828, are used to build the various POD bases whose snapshot energies and projection errors are displayed in Figure~\ref{fig:Bracket-POD-Energy}.  Remarkably, it is again the case here that the Ordinary POD basis is much more expressive than the other constructions, exhibiting projection error which decay nearly to machine precision as basis modes are added.  Moreover, even the Lagrangian construction suffers here, perhaps due to its utilization of only position information.  For the associated ROMs, only reproduction of the snapshots in time is considered for this example, although all ROMs are simulated with and without centering to demonstrate the effect of this choice.  

Results reproducing the motion of the flexible bracket are shown in Figure~\ref{fig:Bracket-T02-NoMC-Errors} for the ROMs without centering and in Figure~\ref{fig:Bracket-T02-Errors} for the centered ROMs.  Notice that the accuracy of each ROM is remarkably insensitive to this choice, with both sets of ROMs displaying similar performance across every choice of basis.  It can also be seen that both the centered and uncentered ROMs preserve the Hamiltonian to the same order, although only the centered ROMs preserve this quantity at the value prescribed by the FOM (not visible because the difference in values is plotted).  Additionally, the results in both Figures show that only the ROMs built with the Ordinary POD basis are capable of exhibiting convergence down to order $10^{-11}$, while other bases are not expressive enough to lower the relative error beyond about $10^{-3}$.  The Lagrangian L-ROM in particular is quite stable but not highly performant.  Moreover, as in the plate example the variationally inconsistent H-ROM exhibits large errors at every basis size, while the consistent H-ROM developed here exhibits smooth convergence as expected by the projection errors in Figure~\ref{fig:Bracket-POD-Energy}.

Finally, several ROM solutions are plotted in in Figure \ref{fig:bracket_zdisp}.  Subfigures (b)--(d) show the $z$--displacement of the flexible bracket problem at the final time $t =2 \times 10^{-2}$ s, along with the FOM $z$--displacement solution in Subfigure (a) and the absolute errors between the ROM and FOM $z$--displacement solutions in Subfigures (e)--(g).  The $z$--displacements in Figure \ref{fig:bracket_zdisp} are scaled by a factor of 10 for visualization purposes, and all ROMs visualized included centering.  The reader can observe some spurious oscillations in the intrusive G-ROM solution (Figure \ref{fig:bracket_zdisp}(b)) with respect to the FOM solution (Figure \ref{fig:bracket_zdisp}(a)).  The situation is worse for the inconsistent H-ROM solution (Figure \ref{fig:bracket_zdisp}(c)), which has completely diverged from the FOM solution.  In contrast, the consistent H-ROM solution (Figure \ref{fig:bracket_zdisp}(d)) matches the FOM solution remarkably well.  These conclusions are supported by the absolute error plots shown in Figure \ref{fig:bracket_zdisp_errs}.

\begin{figure}[tb]
\begin{subfigure}[t]{.24\textwidth}\centering
  \includegraphics[width=\textwidth]{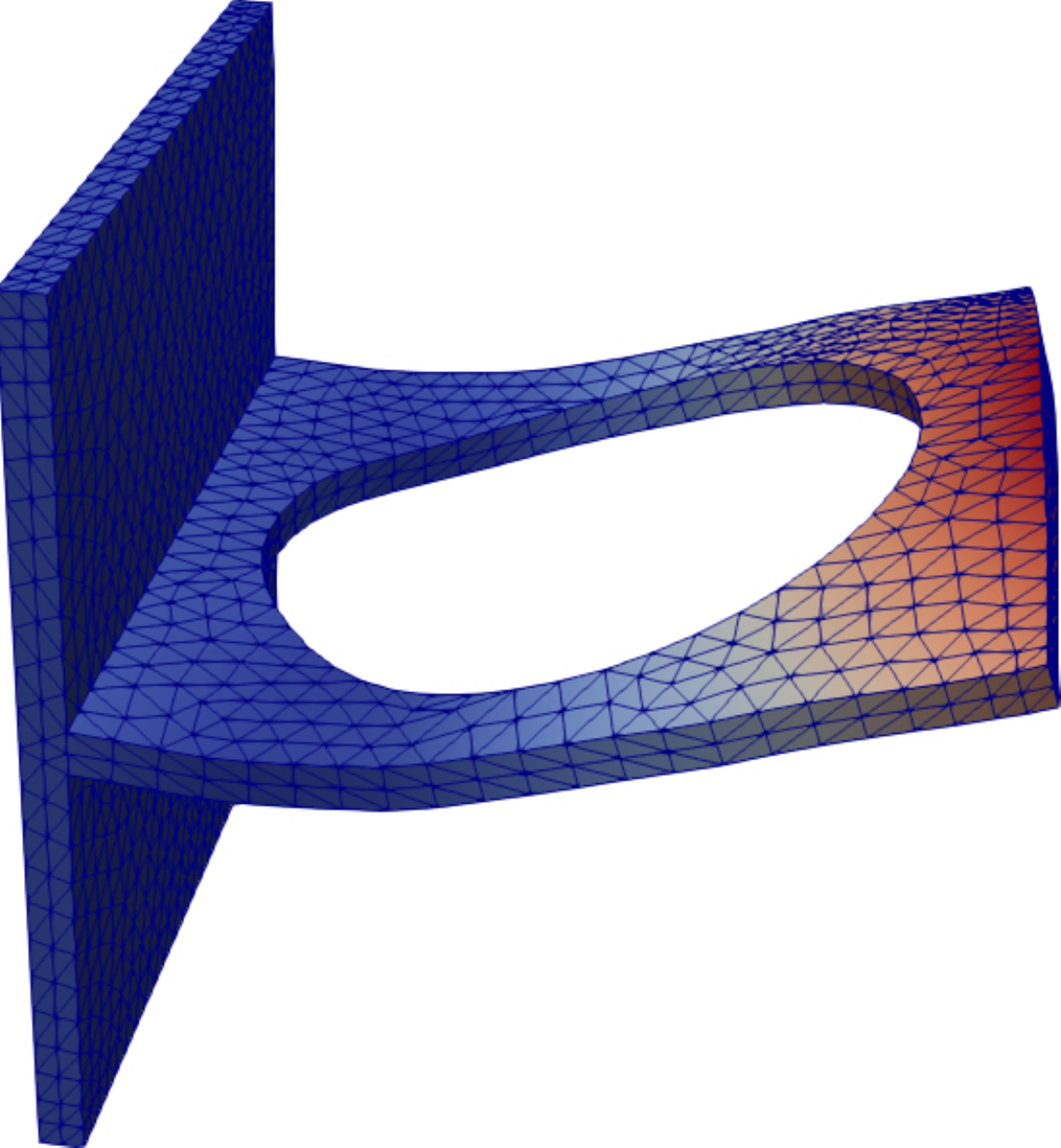}
  \caption{FOM}
\end{subfigure}%
\hfill
\begin{subfigure}[t]{.24\textwidth}\centering
  \includegraphics[width=\textwidth]{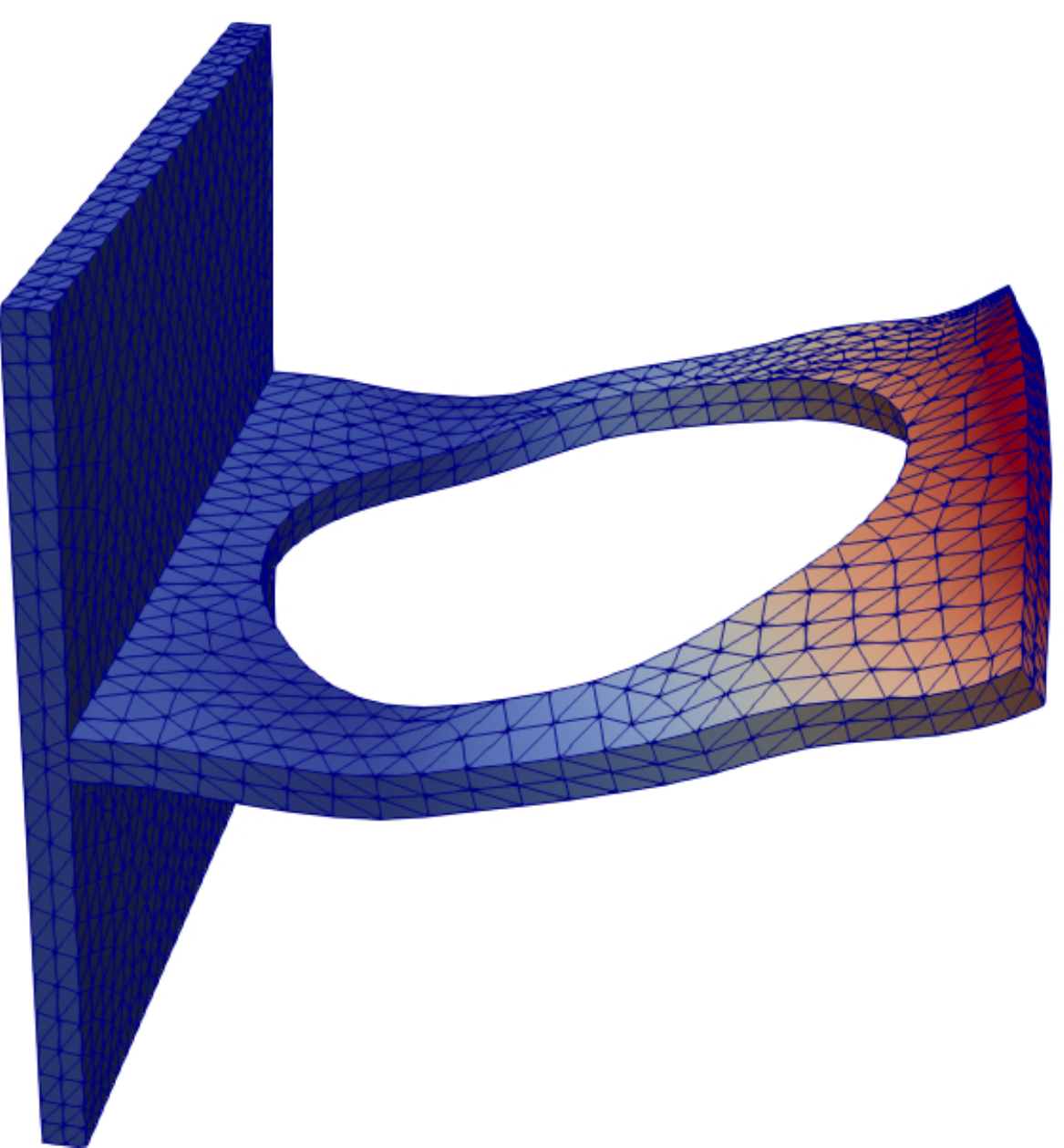}
  \caption{Intrusive G-ROM}
\end{subfigure}%
\hfill
\begin{subfigure}[t]{.24\textwidth}\centering
  \includegraphics[width=\textwidth]{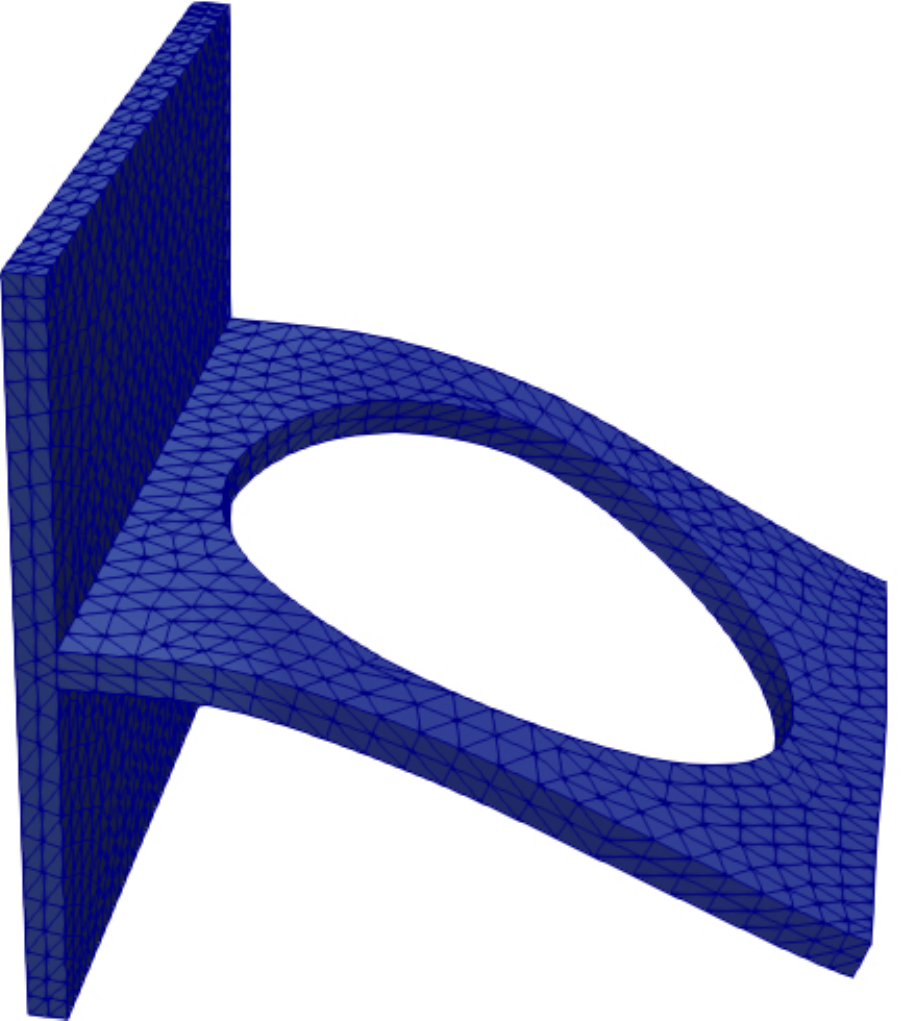}
  \caption{Intrusive H-ROM \\(inconsistent)}
\end{subfigure}%
\hfill
\begin{subfigure}[t]{.24\textwidth}\centering
  \includegraphics[width=\textwidth]{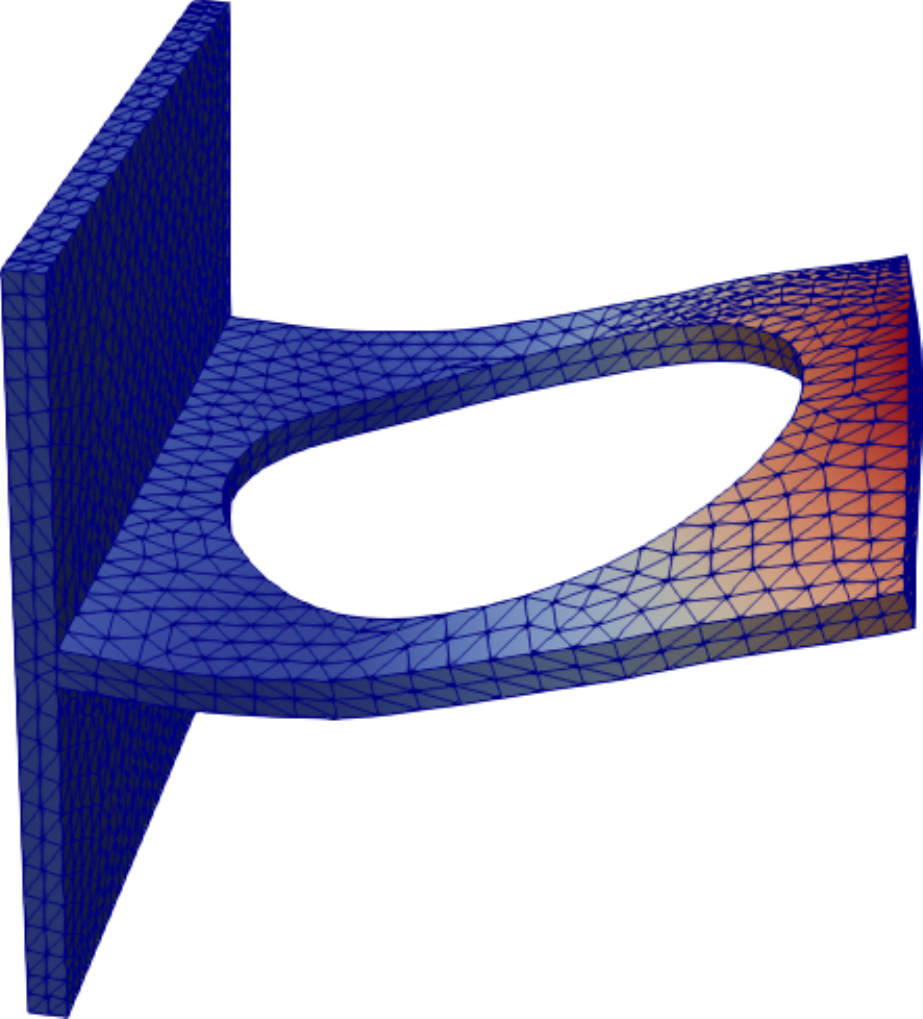}
  \caption{Intrusive H-ROM \\(consistent)}
\end{subfigure}
\caption{$z$--displacement solutions to the flexible bracket problem, scaled by a factor of 10 for plotting purposes, at the final time $t = 2 \times 10^{-2}$ s for the FOM (a) and various intrusive ROMs all calculated using centering (b)--(d). } 
\label{fig:bracket_zdisp}
\end{figure}


\begin{figure}[tb]
\begin{subfigure}[t]{.3\textwidth}\centering
  \includegraphics[width=\textwidth]{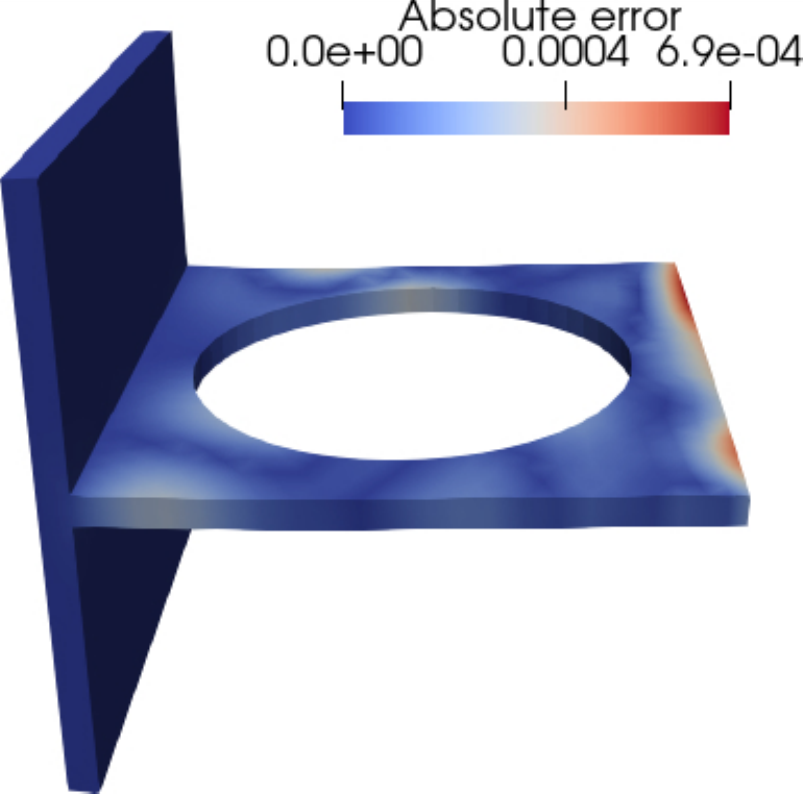}
  \caption{Intrusive G-ROM }
\end{subfigure}%
\hfill
\begin{subfigure}[t]{.3\textwidth}\centering
  \includegraphics[width=\textwidth]{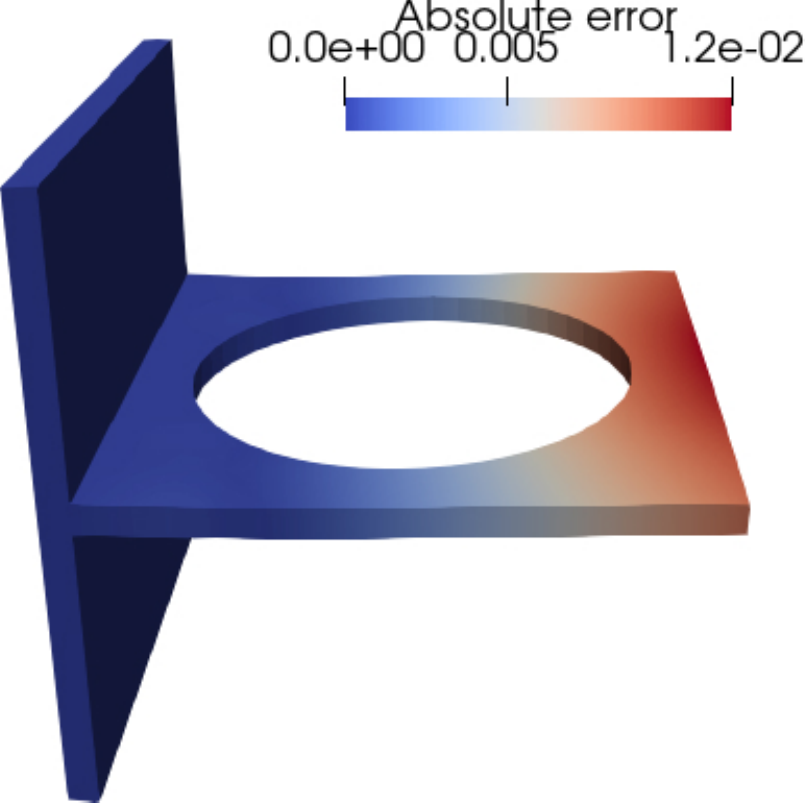}
  \caption{Intrusive H-ROM \\ (inconsistent)}
\end{subfigure}%
\hfill
\begin{subfigure}[t]{.3\textwidth}\centering
  \includegraphics[width=\textwidth]{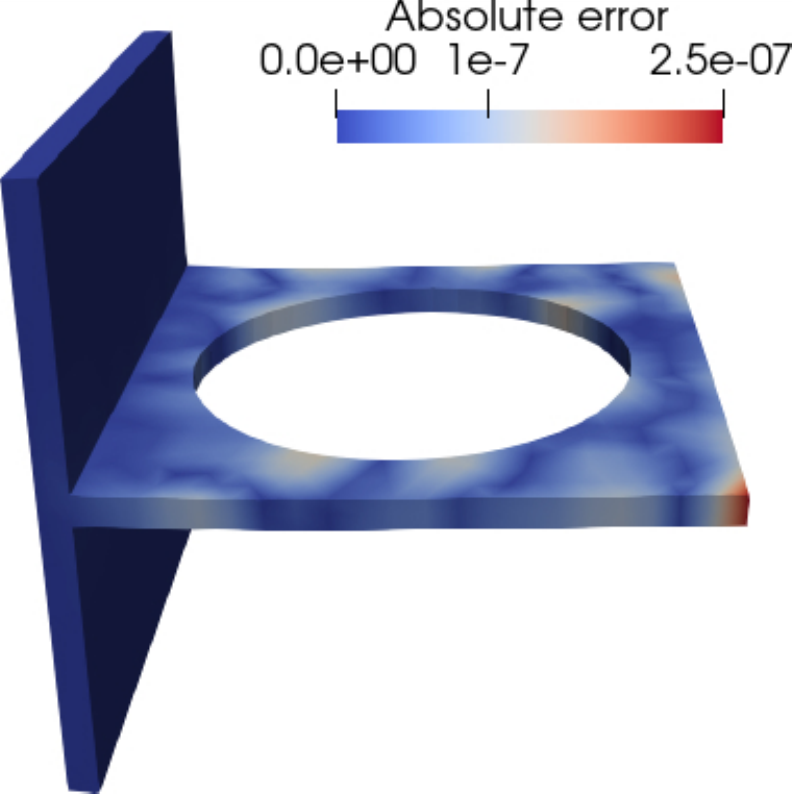}
  \caption{Intrusive H-ROM \\(consistent)}
\end{subfigure}%
\hfill
\caption{Spatial distribution of the absolute errors in the $z$--displacement for various ROM solutions to the flexible bracket problem, evaluated (with centering) at the final time $t=2 \times 10^{-2}$ s. } 
\label{fig:bracket_zdisp_errs}
\end{figure}

\begin{figure}[htb]
    \centering
    \includegraphics[width=\textwidth]{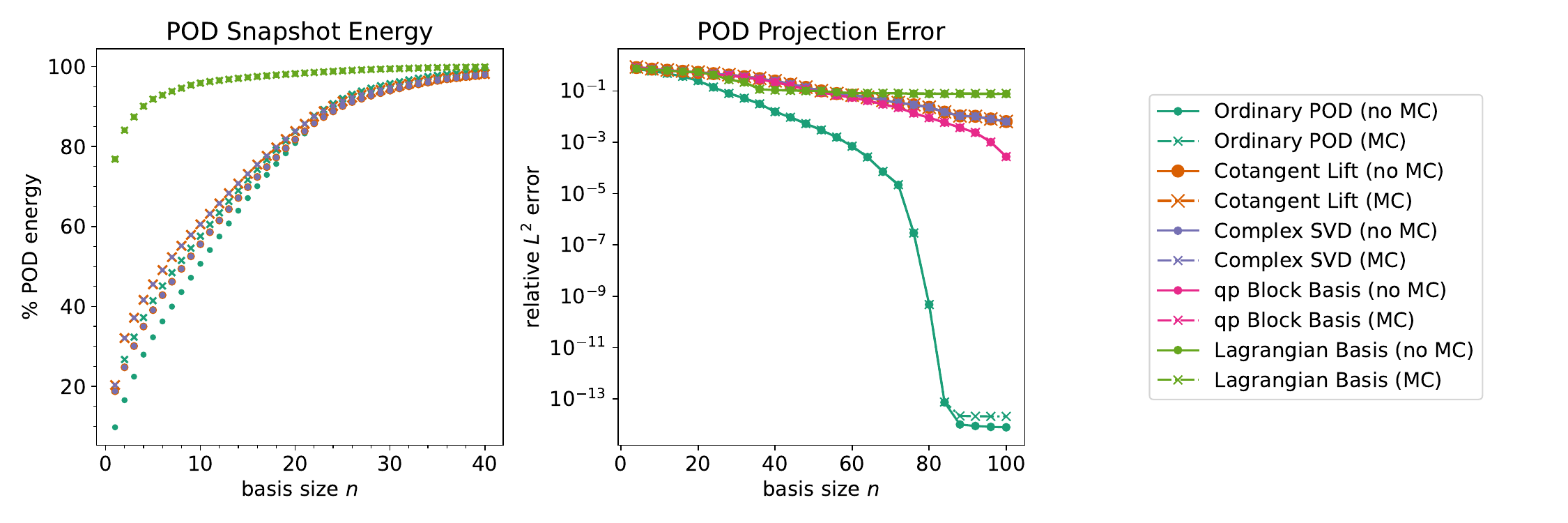}
    \caption{POD basis snapshot energies and projection errors corresponding to the flexible bracket example in Section~\ref{sbsec:bracket}.}
    \label{fig:Bracket-POD-Energy}
\end{figure}


\begin{figure}[htb]
    \centering
    \begin{minipage}{\textwidth}
        \includegraphics[width=\textwidth]{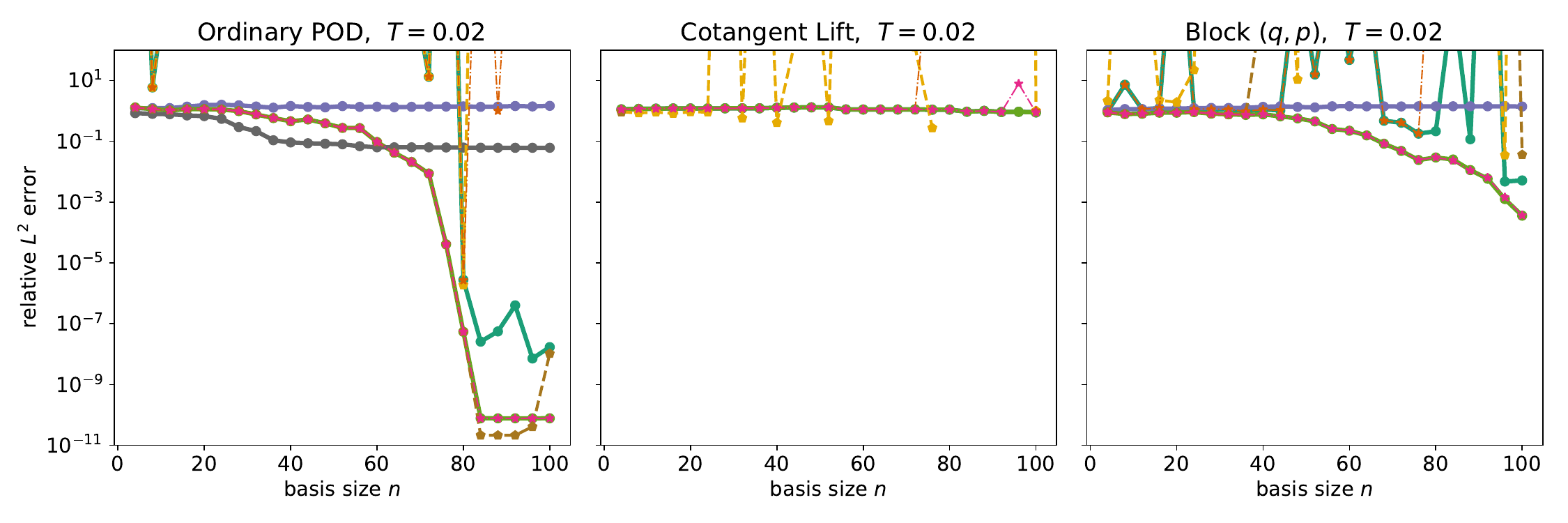}    
    \end{minipage}%
    \\
    \begin{minipage}{\textwidth}
        \includegraphics[width=\textwidth]{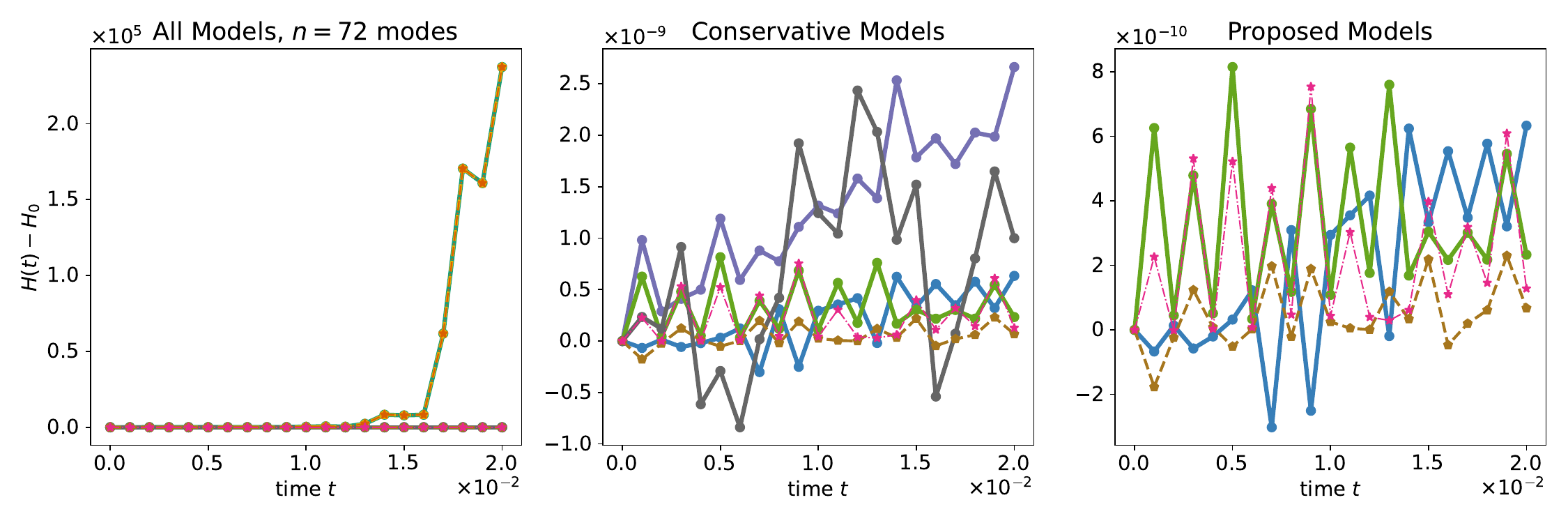}
    \end{minipage}%
    \\
    \begin{minipage}{\textwidth}
    \includegraphics[width=\textwidth]{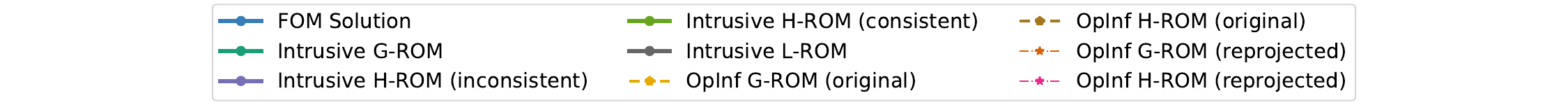}    
    \end{minipage}
    \caption{ROM state errors (top) and Hamiltonian errors (bottom) corresponding to the flexible bracket example in Section~\ref{sbsec:bracket}, without ROM centering.}
    \label{fig:Bracket-T02-NoMC-Errors}
\end{figure}

\begin{figure}[htb]
    \centering
    \begin{minipage}{\textwidth}
        \includegraphics[width=\textwidth]{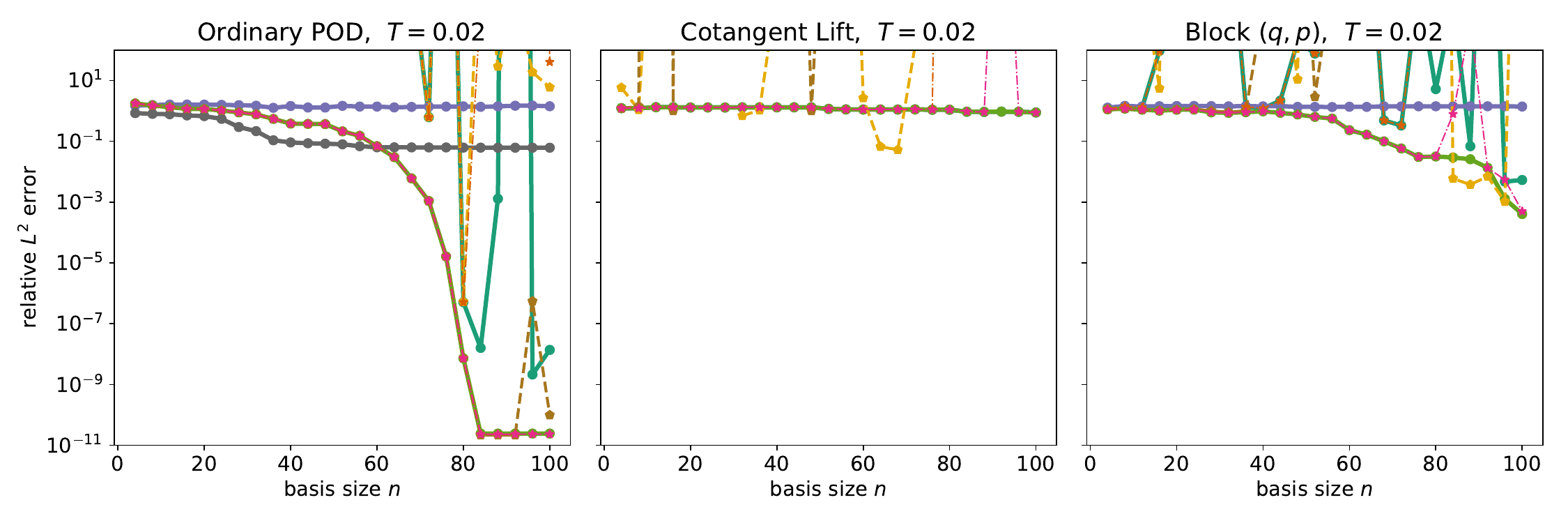}    
    \end{minipage}%
    \\
    \begin{minipage}{\textwidth}
        \includegraphics[width=\textwidth]{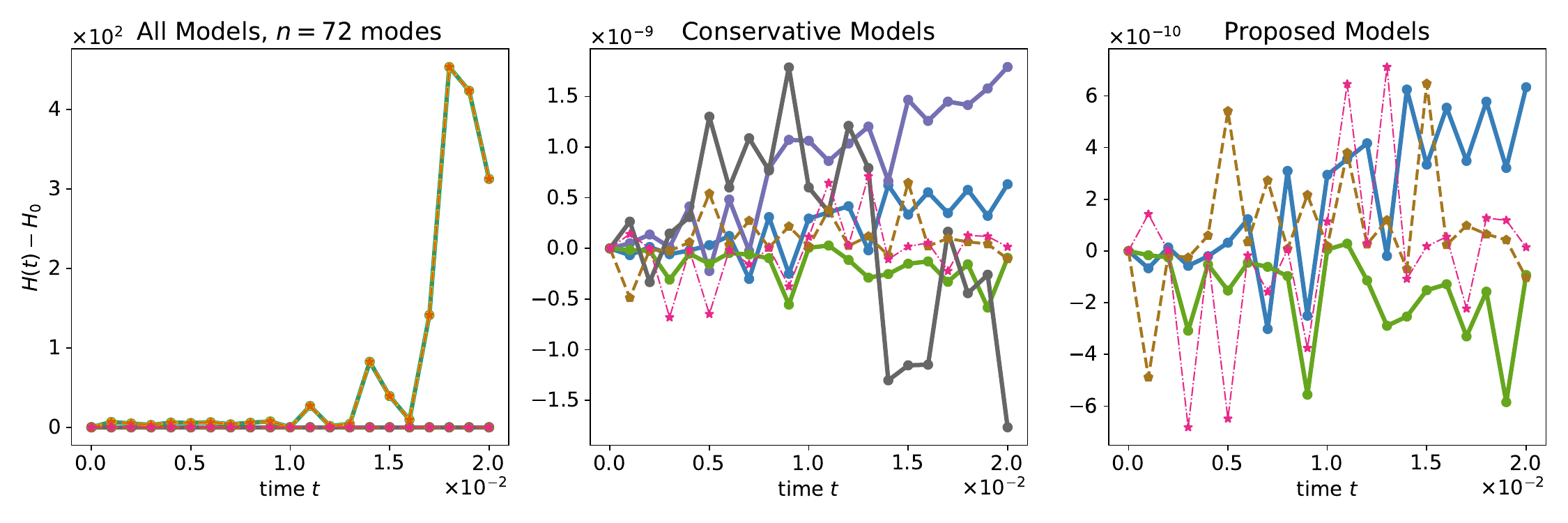}
    \end{minipage}%
    \\
    \begin{minipage}{\textwidth}
    \includegraphics[width=\textwidth]{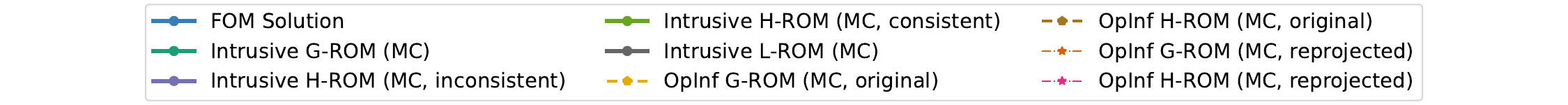}    
    \end{minipage}
    \caption{ROM state errors (top) and Hamiltonian errors (bottom) corresponding to the flexible bracket example in Section~\ref{sbsec:bracket}, with ROM centering.}
    \label{fig:Bracket-T02-Errors}
\end{figure}





\section{Conclusion} \label{sec:conc}
A variationally consistent method for the model reduction of canonical Hamiltonian systems has been proposed which is compatible with any linear orthonormal reduced basis and applicable in both intrusive and nonintrusive settings.  In addition to exhibiting exact energy conservation and symplecticity, the proposed method does not require additional considerations beyond Petrov-Galerkin orthogonality and has been shown to produce accurate solutions in situations where previous methods cannot.  Error estimates have been proven which interpret the error in the proposed method in terms of a balance between the usual basis projection error and an additional term arising from the failure of the reduced system to be canonically Hamiltonian.  Since canonical Hamiltonian systems of appropriate regularity necessarily have a natural Lagrangian counterpart which is Legendre-dual, future work will compare and contrast the approach presented here with Lagrangian ROMs based on the second-order equations of motion.  It is further interesting to consider the extension of these ideas to nonlinear problems, quadratic or nonlinear bases, and noncanonical Hamiltonian systems where there is no Lagrangian analogue.  It is hoped that the work presented here will provide a useful starting point for researchers who wish to deploy ROMs on large-scale Hamiltonian systems with complex physics and realistic material parameters.  


\section*{Acknowledgements}
Sandia National Laboratories is a multimission laboratory managed and operated by National Technology \& Engineering Solutions of Sandia, LLC, a wholly owned subsidiary of Honeywell International Inc., for the U.S. Department of Energy’s National Nuclear Security Administration under contract DE-NA0003525. This paper describes objective technical results and analysis. Any subjective views or opinions that might be expressed in the paper do not necessarily represent the views of the U.S. Department of Energy or the United States Government. This article has been co-authored by an employee of National Technology \& Engineering Solutions of Sandia, LLC under Contract No. DE-NA0003525 with the U.S. Department of Energy (DOE). The employee owns all right, title and interest in and to the article and is solely responsible for its contents. The United States Government retains and the publisher, by accepting the article for publication, acknowledges that the United States Government retains a non-exclusive, paid-up, irrevocable, world-wide license to publish or reproduce the published form of this article or allow others to do so, for United States Government purposes. The DOE will provide public access to these results of federally sponsored research in accordance with the DOE Public Access Plan https://www.energy.gov/downloads/doe-public-access-plan.  The work of Anthony Gruber was supported by the John von Neumann fellowship at Sandia National Laboratories.  The work of Irina Tezaur was supported by the U.S. Presidential Early Career Award for Scientists and Engineers (PECASE) and the Multifaceted Mathematics for Predictive Digital Twins (M2dt) project under the U.S. Department of Energy's Office of Advanced Scientific Computing Research.  The authors thank Eric Parish for helpful conversations on variational consistency in ROM, Elise Walker for careful proofreading and suggestions regarding the presentation of Proposition~\ref{prop:nondegen}, and Alejandro Mota for helping with the setup of the 3D elasticity test cases presented herein and simulated using the Albany-LCM code.

\bibliographystyle{abbrv}
\bibliography{biblio.bib}

\end{document}